\documentclass[11pt]{article}

\usepackage{amsmath}
\usepackage{amssymb}
\usepackage{amsfonts}
\usepackage{amsthm}
\usepackage{array}
\usepackage{bm}
\usepackage{enumitem}
\usepackage{stmaryrd}

\usepackage{amsxtra}
\usepackage{array}
\usepackage{mathrsfs}
\usepackage{slashed}
\usepackage{xcolor}

\usepackage{tikz-cd}

\newcolumntype{C}[1]{>{\centering\hspace{0pt}}p{#1}}
\usepackage[margin=1in]{geometry}

\newtheorem{thm}{Theorem}[section]
\newtheorem{prop}[thm]{Proposition}
\newtheorem{lem}[thm]{Lemma}
\newtheorem{cor}[thm]{Corollary}

\theoremstyle{definition}

\theoremstyle{definition}
\newtheorem*{rmk}{Remark}

\usepackage[nottoc]{tocbibind}
\numberwithin{equation}{section}

\usepackage[hidelinks]{hyperref}

\title{The Second Variation for Null-Torsion \\ Holomorphic Curves in the 6-Sphere}
\author{Jesse Madnick}
\date{December 2021}

\newcommand{\Addresses}
{{   \bigskip
  \textsc{National Center for Theoretical Sciences} \par\nopagebreak
\noindent   \textsc{National Taiwan University}\par\nopagebreak
\noindent   \textsc{Taipei, Taiwan}\par\nopagebreak
\noindent  \textit{E-mail address}: \texttt{jmadnick@ncts.ntu.edu.tw}
}}

\begin{document}

\maketitle

\begin{abstract}
In the round 6-sphere, null-torsion holomorphic curves are fundamental examples of minimal surfaces.  This class of minimal surfaces is quite rich: By a theorem of Bryant, extended by Rowland, every closed Riemann surface may be conformally embedded in the round $6$-sphere as a null-torsion holomorphic curve.

In this work, we study the second variation of area for compact null-torsion holomorphic curves $\Sigma$ of genus $g$ and area $4\pi d$, focusing on the spectrum of the Jacobi operator.  We show that if $g \leq 6$, then the multiplicity of the lowest eigenvalue $\lambda_1 = -2$ is equal to $4d$.  Moreover, for any genus, we show that the nullity is at least $2d + 2 - 2g$.  These results are likely to have implications for the deformation theory of asymptotically conical associative $3$-folds in $\mathbb{R}^7$, as studied by Lotay. 

\end{abstract}


\section{Introduction}

\subsection{Background: Minimal Surfaces in Spheres}

\indent \indent Let $\Sigma^2$ denote a closed orientable surface.  In a Riemannian manifold $(M, \langle \cdot, \cdot \rangle)$, an immersed surface $u \colon \Sigma^2 \to M$ is called a \textit{minimal surface} if every variation $u_t \colon \Sigma^2 \to M$ of $u_0 = u$ satisfies $\left.\frac{d}{dt}\right|_{t = 0} \text{Area}(u_t) = 0$.  That is, minimal surfaces are critical points of the area functional, but not necessarily global minimizers of it.  The extent to which a minimal surface fails to be area-minimizing to second order can be measured by the \textit{second variation of area}, which takes the form
\begin{equation} \label{eq:StandardSecondVar}
\left.\frac{d^2}{dt^2}\right|_{t = 0}\text{Area}(u_t) = \int_\Sigma \left\langle \mathcal{L}\eta, \eta \right\rangle\!,
\end{equation}
where $\eta := \left.\frac{d}{dt}\right|_{t = 0} u_t$ is a normal variation vector field, and where $\mathcal{L} \colon \Gamma(N\Sigma) \to \Gamma(N\Sigma)$ is the \textit{Jacobi operator} of the minimal surface $u$.  We will recall the standard expression for $\mathcal{L}$ in $\S$4.

In view of the second variation formula (\ref{eq:StandardSecondVar}), it is of fundamental interest to understand the Jacobi operator of a minimal surface, which in turn motivates the study its spectrum.  Indeed, recalling that $\mathcal{L}$ is strongly elliptic \cite[$\S$I.9]{lawson80lectures}, it may be diagonalized with real eigenvalues
$$\lambda_1 < \lambda_2 < \cdots < \lambda_{s} < 0 = \lambda_{s+1} < \lambda_{s+2} < \cdots \to \infty$$
of finite multiplicities
$$m_1,\, m_2,\, \ldots, m_s, m_{s+1}, \ldots$$
Certain well-known invariants of minimal surfaces may be phrased in terms of this spectrum.  For example, the \textit{Morse index} and \textit{nullity} are, respectively,
\begin{align*}
\text{Ind}(u) & = m_1 + \cdots + m_s & \text{Nullity}(u) & = m_{s+1}.
\end{align*}
A minimal surface is said to be \textit{stable} if $\lambda_1 \geq 0$ and \textit{unstable} if $\lambda_1 < 0$.

\pagebreak

\indent In general, computing the spectrum of $\mathcal{L}$ is extremely difficult.  There seem to be very few examples of minimal surfaces whose Jacobi spectra are known explicitly.  In view of this, geometers instead seek to estimate the eigenvalues $\lambda_j$ and (sums of) multiplicities $m_j$ in terms of more computable geometric and topological quantities.  Still, obtaining such bounds is a non-trivial task.  Even in the classical case of (non-compact) minimal surfaces in $\mathbb{R}^3$, several outstanding open problems remain: see, for example, the excellent survey \cite{chodosh21}. \\

\indent Our focus will be on compact orientable minimal surfaces (without boundary) in round spheres $M = \mathbb{S}^n$ (of constant curvature $1$) with $n \geq 3$.  Pioneering work in this subject was carried out in the late 1960's by, for example, Calabi \cite{calabi1967minimal}, Chern \cite{chern1970minimal}, Simons \cite{simons68}, and Lawson \cite{lawson1970complete}.  In particular, Simons showed \cite[Lemma 5.1.4]{simons68} that all compact minimal surfaces in $\mathbb{S}^n$ have $\lambda = -2$ as an eigenvalue of $\mathcal{L}$, and hence are unstable.  He also established the lower bounds
\begin{align*}
\text{Ind}(u) & \geq n-2 & \text{Nullity}(u) & \geq 3(n-2).
\end{align*}
In both estimates, equality holds if and only if $u$ is the totally-geodesic $\mathbb{S}^2$.

\indent In the case of $n = 3$, Urbano \cite{urbano90} improved Simons' bound, showing that non-totally-geodesic minimal surfaces satisfy $\text{Ind}(u) \geq 5$, with equality if and only if $u$ is the Clifford torus.  This characterization of the Clifford torus was an important ingredient in Marques' and Neves' resolution of the Willmore conjecture \cite{marques2014min}.

\indent In the case $n = 4$, Micallef and Wolfson \cite{micallef-wolfson} proved that minimal surfaces in $\mathbb{S}^4$ of area $A$ satisfy
$$\text{Ind}(u) \geq \frac{1}{2}\left(\frac{A}{\pi} - \chi(\Sigma)\right)\!,$$
where $\chi(\Sigma) = 2 - 2g$ is the Euler characteristic.  Recently, motivated by potential applications to the generalized Willmore conjecture in $\mathbb{S}^n$, Kusner and Wang \cite{kusner-wang} proved that minimal surfaces of genus $g = 1$ in $\mathbb{S}^4$ satisfy $\text{Ind}(u) \geq 6$, with equality if and only if $u$ is a Clifford torus in a totally-geodesic $\mathbb{S}^3$.

In a different direction, if one restricts attention to the class of \textit{superminimal} surfaces in $\mathbb{S}^4$, the beautiful paper of Montiel-Urbano \cite{montiel-urbano} provides remarkably precise information.  They show that superminimal surfaces in $\mathbb{S}^4$ have lowest eigenvalue $\lambda_1 = -2$ and satisfy 
\begin{align} \label{eq:MU}
\text{Ind}(u) & = m_1 = \frac{A}{\pi} - \chi(\Sigma) & \text{Nullity}(u) = m_2 & \geq \frac{A}{\pi} + \chi(\Sigma).
\end{align}
Moreover, if $g = 0$ or $g = 1$, then equality holds in the nullity estimate.  In fact, $\text{Ind}(u) \geq 10$, with equality if and only if $u$ is a (twistor deformation of a) Veronese surface.  The formulas (\ref{eq:MU}) for superminimal surfaces in $\mathbb{S}^4$ were the primary inspiration for this work.

\indent In even dimensions $n = 2k$, Karpukhin \cite{karpukhin2021index} has recently shown that a linearly full minimal surface in $\mathbb{S}^{2k}$ of genus $g = 0$ and area $A = 4\pi d$ has index
$$\text{Ind}(u) \geq 2(k-1)(2d - [\sqrt{8d+1}]_{\text{odd}} + 2),$$
where $[x]_{\text{odd}}$ is the largest odd integer not exceeding $x$.

\subsection{Background: Holomorphic Curves in the $6$-Sphere}

\indent \indent Among all spheres $\mathbb{S}^n$ with $n \geq 3$, the $6$-sphere is the only one that admits an almost-complex structure.  In this work, we will equip $\mathbb{S}^6$ with its standard almost-complex structure $\widetilde{J} \colon T\mathbb{S}^6 \to T\mathbb{S}^6$.  This almost-complex structure is compatible with the round metric, and arises from viewing $\mathbb{S}^6 \subset \mathbb{R}^7 = \text{Im}(\mathbb{O})$ in the imaginary octonions, as we will recall in $\S$\ref{sect:SU3-Structure}.  Having chosen $\widetilde{J}$, the $6$-sphere now admits a distinguished class of surfaces.  That is, a \textit{holomorphic curve} is a surface $u \colon \Sigma^2 \to \mathbb{S}^6$ whose tangent spaces are $\widetilde{J}$-invariant:
$$\widetilde{J}(T_p\Sigma) = T_p\Sigma, \ \ \forall p \in \Sigma.$$
It is easy to show that holomorphic curves in $\mathbb{S}^6$ are (unstable) minimal surfaces.

\indent In a remarkable 1982 paper, Bryant \cite{bryant82} studied holomorphic curves in $\mathbb{S}^6$ by means of a  ``holomorphic Frenet frame," which we discuss in $\S$\ref{sec:MovingFramesHoloCurves}.  Essentially, this amounts to a decomposition of the vector bundle of $(1,0)$-vectors along $u(\Sigma)$ into complex line subbundles
\begin{equation} \label{eq:TNB}
u^*(T^{1,0}\mathbb{S}^6) \simeq L_T \oplus L_N \oplus L_B.
\end{equation}
Crucially, each of the bundles $L_T, L_N, L_B$ carries a natural holomorphic structure, though the isomorphism (\ref{eq:TNB}) generally only holds in the smooth (not holomorphic) category.  By analogy with the classical case of curves in $\mathbb{R}^3$, one can extract two basic invariants: a second-order invariant (``curvature") that is essentially the second fundamental form of the immersion, and a third-order invariant (``torsion") that is rather more subtle.  Bryant encodes the torsion as a holomorphic section
$$\Phi_{\text{I\!I\!I}} \in H^0(L_T^* \otimes L_N^* \otimes L_B),$$
and defines a holomorphic curve to be \textit{null-torsion} if $\Phi_{\text{I\!I\!I}} \equiv 0$ on $\Sigma$.  It is not hard to show that every holomorphic curve of genus $g = 0$ is null-torsion. \\
\indent It turns out that the null-torsion condition is equivalent to the holomorphicity of the \textit{binormal Gauss map} $b_u \colon \Sigma \to \mathbb{CP}^6$, the map sending a point $p \in \Sigma$ to its binormal real $2$-plane in $T_p\mathbb{S}^6 \subset \mathbb{R}^7$ (viewed as a complex line in $\mathbb{C}^7$).  From this fact, together with the Wirtinger Theorem, it follows that the area $A$ of a null-torsion holomorphic curve is quantized.  That is,
$$A = 4\pi d,$$
where $d \in \mathbb{Z}^+$ is the degree of the binormal Gauss map.  Aside from the totally-geodesic $2$-sphere (which has $d = 1$), all null-torsion holomorphic curves have $d \geq 6$.  The moduli space of genus zero holomorphic curves in $\mathbb{S}^6$ of a fixed degree $d \geq 6$ has been studied by Fern\'{a}ndez \cite{fernandez15space}.

\indent In \cite{bryant82}, Bryant derived a Weierstrass representation formula for null-torsion holomorphic curves.  Using this formula, together with an algebro-geometric argument, he proved a striking result: every closed Riemann surface admits a conformal branched immersion into $\mathbb{S}^6$ as a null-torsion holomorphic curve.  This was sharpened by Rowland \cite{rowland} in his 1999 Ph.D. thesis, who improved ``branched immersion" to ``smooth embedding."  The upshot is that, while the generic holomorphic curve is not null-torsion, the class of null-torsion curves is nevertheless extremely rich. \\

\indent Since Bryant's 1982 paper, there have been several interesting studies of holomorphic curves in the $6$-sphere.  For example, Sekigawa \cite{sekigawa83almost} classified the constant-curvature examples, Ejiri \cite{ejiri1986equivariant} classified the $\text{U}(1)$-invariant examples, and Hashimoto \cite{hashimoto2004} obtained beautiful explicit examples of one-parameter deformations.  Bolton, Vrancken, and Woodward \cite{bolton94} studied holomorphic curves by using harmonic sequences, and showed that every holomorphic curve in $\mathbb{S}^6$ can only be full in a totally-geodesic $\mathbb{S}^2$, $\mathbb{S}^5$, or else the entire $\mathbb{S}^6$.  This is by no means a complete list of references; we refer the interested reader to the books of Chen \cite[$\S$19.1-19.2]{chen00riemannian} and Joyce \cite[$\S$12.2]{joyce2007} for more.

\indent Finally, we note that the study of holomorphic curves in $\mathbb{S}^6$ forms part of the larger study of holomorphic curves in nearly-K\"{a}hler $6$-manifolds.  For example, holomorphic curves in $\mathbb{CP}^3$ have been studied by Xu \cite{xu10} and Aslan \cite{aslan2101}, and in $\mathbb{S}^3 \times \mathbb{S}^3$ by Bolton, Dioos, and Vrancken \cite{bolton2015almost}.  In fact, holomorphic curves in nearly-K\"{a}hler $6$-manifolds are precisely the links of associative cones in conical $\text{G}_2$-manifolds, and thereby serve as models for conically singular associative $3$-folds.   This relationship makes holomorphic curves objects of fundamental interest in $\text{G}_2$-geometry.

\subsection{Main Results}

\indent \indent In this work, we consider the Jacobi spectra of null-torsion holomorphic curves in $\mathbb{S}^6$.  Perhaps the most basic question is: What is the multiplicity $m_1$ and value $\lambda_1$ of the lowest eigenvalue of $\mathcal{L}$?  In the early 1980's, Ejiri \cite{ejiri83} considered this question in the context of superminimal surfaces in $\mathbb{S}^{2n}$, showing that $\lambda_1 = -2$.  (Although his results are stated for minimal $2$-spheres in $\mathbb{S}^{2n}$, most of Ejiri's arguments apply without change to the larger class of superminimal surfaces.)  Furthermore, equipping the normal bundle with a certain holomorphic structure, which we call $\overline{\partial}^\nabla$, he showed that the $\lambda_1$-eigenspace of $\mathcal{L}$ may be identified with the space of holomorphic normal vector fields:
$$\{\eta \in \Gamma(N\Sigma) \colon \mathcal{L}\eta = -2\eta\} \cong \{\text{solutions of }\overline{\partial}^\nabla \xi = 0\}.$$
The Riemann-Roch Theorem then implies that
$$m_1 \geq \frac{A}{\pi} + (n-3)\chi(\Sigma).$$
Ejiri also observed that equality holds in the case of genus $g = 0$, essentially by an application of Grothendieck's classification of holomorphic vector bundles on $\mathbb{S}^2 = \mathbb{CP}^1$.

\indent Now, since null-torsion holomorphic curves in $\mathbb{S}^6$ are, in particular, superminimal surfaces, Ejiri's results imply that they satisfy $\lambda_1 = -2$ and
\begin{equation} \label{eq:EjiriLowerBound}
m_1 \geq \frac{A}{\pi}.
\end{equation}
Our first result is that, in fact, equality holds for genus $g \leq 6$:

\begin{thm} \label{thm:M1}
Let $u \colon \Sigma \to \mathbb{S}^6$ be a null-torsion holomorphic curve of genus $g$ and area $A = 4\pi d$.
If $g \leq 6$ (or, more generally, if $g < \frac{1}{2}(d + 2)$), then the first multiplicity $m_1$ of the Jacobi operator is:
\begin{equation*}
m_1 = \frac{A}{\pi} = 4d.
\end{equation*}
\end{thm}

\indent Where minimal surfaces of high genus ($g \geq 1$) and high codimension (at least $2$) in round spheres are concerned, the only explicit formulas for $m_1$ that the author knows are Montiel and Urbano's result (\ref{eq:MU}) and Theorem \ref{thm:M1} above.  Our argument makes crucial use of the particular geometry of null-torsion holomorphic curves.  In outline, the idea is the following. We will equip the normal bundle with a \textit{second} holomorphic structure, called $\overline{\partial}^D$, that arises naturally from the nearly-K\"{a}hler structure on $\mathbb{S}^6$.  Letting $S$ denote the difference tensor $S\xi := \overline{\partial}^\nabla\xi - \overline{\partial}^D \xi$, the Cauchy-Riemann system $\overline{\partial}^\nabla\xi = 0$ is equivalent to
\begin{equation} \label{eq:EquivCR}
\overline{\partial}^D \xi = -S\xi.
\end{equation}
It turns out that the system (\ref{eq:EquivCR}) decouples, yielding an easy upper bound for the dimension of the solution space, which proves the theorem.  Our second result is a lower bound on the nullity, valid for all genera:

\begin{thm} \label{thm:Nullity}
Let $u \colon \Sigma \to \mathbb{S}^6$ be a null-torsion holomorphic curve of genus $g$ and area $A = 4\pi d$.  Then the nullity of its Jacobi operator satisfies
\begin{equation} \label{eq:NullityBound}
\mathrm{Nullity}(u) \geq 2d + \chi(\Sigma).
\end{equation}
\end{thm}
\indent Here, our argument is not original.  Indeed, we closely follow the calculations in Montiel and Urbano's study \cite{montiel-urbano} of superminimal surfaces in self-dual Einstein $4$-manifolds.  The idea of the proof is to identify a certain subspace of
$$\text{Null}(u) := \{\eta \in \Gamma(N\Sigma) \colon \mathcal{L}\eta = 0\}$$
with the space of holomorphic sections of a certain line bundle whose dimension can be estimated (and for genus $g \leq 6$, computed explicitly) by Riemann-Roch.  Note that, since we only consider a subspace of $\text{Null}(u)$, our bound (\ref{eq:NullityBound}) is almost certainly not sharp.  On the other hand, it appears that most of the argument extends without change to the general case of superminimal surfaces in any even-dimensional sphere $\mathbb{S}^{2n}$, providing an avenue for further inquiry.

\subsection{Open Questions}

\begin{enumerate}
\item Let $u \colon \Sigma^2 \to \mathbb{S}^{2n}$ be a compact orientable superminimal surface of genus $g$.  Is it always the case that Ejiri's lower bound is satisfied:
$$m_1 = \frac{A}{\pi} + (n-3)\chi(\Sigma)?$$
Ejiri \cite{ejiri83} has proven this for genus $g = 0$, while Montiel-Urbano \cite{montiel-urbano} has proven this for $n = 2$.  Our Theorem \ref{thm:M1} establishes this in the special case where $n = 3$, $g \leq 6$, and the superminimal surface is holomorphic.
\item Holomorphic curves may be studied in any nearly-K\"{a}hler $6$-manifold.  What can be said about the Jacobi spectrum in that generality?
\item In the $6$-sphere: Can one establish a lower bound on the second eigenvalue $\lambda_2$?  As a first step, it would be instructive to understand the spectrum of the Boruvka sphere, the unique holomorphic curve of constant curvature $K = \frac{1}{6}$.  We show in Proposition \ref{thm:Lambda2} that the Boruvka sphere satisfies $\lambda_2 \geq -\frac{5}{3}$.  Further, Karpukhin \cite[Theorem 1.7]{karpukhin2021index} has estimated its Morse index as $m_1 + \cdots + m_s = \text{Ind}(u) \geq 36$, and Ejiri's result \cite{ejiri83} gives $m_1 = 24$, implying $m_2 + \cdots + m_s \geq 12$, so $\lambda_2 < 0$. 
\end{enumerate}

\subsection{Organization}

\indent \indent In $\S$2, we recall basic facts and formulas regarding minimal surfaces in $\mathbb{S}^6$, holomorphic curves in $\mathbb{S}^6$, and holomorphic vector bundles over Riemann surfaces.  This section is largely to establish conventions and experts may wish to skip it. 

In $\S$3.1 and $\S$3.2, we set up the moving frame for holomorphic curves in the $6$-sphere.  Our discussion is essentially a summary of \cite[$\S$4]{bryant82}, though our notation is quite different.  In $\S$3.3, we take a closer look at \textit{null-torsion} holomorphic curves, culminating in Proposition \ref{thm:HoloSections}, which counts the holomorphic sections of $L_N$ and $L_B^*$, and Proposition \ref{thm:GenusDegree}, which justifies our tacit parenthetical claim in Theorem \ref{thm:M1} that $g \leq 6$ implies $g < \frac{1}{2}(d+2)$.

Section 3.4 is at the heart of Theorem \ref{thm:M1}.  The purpose of $\S$3.4 is to explain how the normal bundle of a null-torsion holomorphic curve may naturally be equipped with \textit{three} different holomorphic structures, which we call $\overline{\partial}^{\text{SU}}$, $\overline{\partial}^\nabla$, and $\overline{\partial}^D$.  The operator $\overline{\partial}^{\text{SU}}$ relates to the deformation theory of asymptotically conical associative $3$-folds in $\mathbb{R}^7$, as shown by Lotay \cite{lotay11asymp}, while $\overline{\partial}^\nabla$ relates to the $(-2)$- and $0$-eigenspaces of the Jacobi operator $\mathcal{L}$.  However, it is with respect to $\overline{\partial}^D$ that the normal bundle splits \textit{holomorphically}, which aids in the decoupling of (\ref{eq:EquivCR}).

\indent In $\S$4, we begin our study of the Jacobi operator of null-torsion holomorphic curves.  Sections 4.1 and 4.2 establish the lower bound (\ref{eq:EjiriLowerBound}), while $\S$4.3 proves Theorem  \ref{thm:M1} by analyzing (\ref{eq:EquivCR}).  Finally, in $\S$5.2, we reduce Theorem 1.2 to a claim (Proposition \ref{thm:GSurj}) about the image of a certain linear Cauchy-Riemann type operator, and in $\S$5.3-$\S$5.4 we establish Proposition \ref{thm:GSurj}. \\

\noindent \textbf{Acknowledgements:} This work benefited from clarifying conversations with Benjamin Aslan, Gavin Ball, Robert Bryant, Bang-Yen Chen, Mikhail Karpukhin, Hsueh-Yung Lin, Jason Lotay, and David Wen.  I thank Gorapada Bera for his careful reading of an earlier version of this preprint, and thank Da Rong Cheng, Shubham Dwivedi, Spiro Karigiannis, and Chung-Jun Tsai for their interest and encouragement. \\
\indent This work was completed during the author's postdoctoral fellowship at the National Center for Theoretical Sciences (NCTS) at National Taiwan University.  I thank the Center for their support.


\section{Preliminaries}

\indent \indent In this brief section, we recall basic facts about minimal surfaces in $\mathbb{S}^6$, holomorphic curves in $\mathbb{S}^6$, and holomorphic vector bundles.  This section primarily serves to fix notation and conventions.

\subsection{Minimal Surfaces in $\mathbb{S}^6$}

\indent \indent Let $u \colon \Sigma^2 \to \mathbb{S}^6$ be an immersed surface in the round $6$-sphere of constant curvature $1$.  Let $\langle \cdot, \cdot \rangle$ denote the round metric on $\mathbb{S}^6$ and let $\overline{\nabla} \colon \Gamma(T\mathbb{S}^6) \to  \Omega^1(\mathbb{S}^6) \otimes \Gamma(T\mathbb{S}^6)$ denote the Levi-Civita connection.  As usual, we split
$$u^*(T\mathbb{S}^6) = T\Sigma \oplus N\Sigma$$
into tangential and normal parts.  For $X,Y \in \Gamma(T\Sigma)$ and $N \in \Gamma(N\Sigma$), we have
\begin{align*}
\overline{\nabla}_XY & = \nabla^\top_XY + \text{I\!I}(X,Y) \\
\overline{\nabla}_XN & = W_XN + \nabla^\perp_XN
\end{align*}
where $\nabla^\top$ is the Levi-Civita connection on $\Sigma$, where $\nabla^\perp$ is the normal connection, where $\text{I\!I}$ is the second fundamental form, and where $W$ is the shape operator.  Recall the Weingarten equation
$$\langle W_XN, Y \rangle = -\langle \text{I\!I}(X,Y), N \rangle.$$
The curvature tensors of $\overline{\nabla}, \nabla^\top, \nabla^\perp$ will be denoted $\overline{R}$, $R^\top$, $R^\perp$, respectively.  We will often use the notation $\overline{R}_{XY}Z := \overline{R}(X,Y, Z, \cdot)$, and similarly for $R^\top$ and $R^\perp$. \\

\indent Suppose now that $u \colon \Sigma^2 \to \mathbb{S}^6$ is a minimal surface.  Let $(e_1, \ldots, e_6)$ be a local orthonormal frame with $e_1, e_2 \in T\Sigma$ and $e_3, e_4, e_5, e_6 \in N\Sigma$.  We recall the \textit{Gauss equation}
\begin{equation} \label{eq:GaussEq}
1 = K + \Vert \text{I\!I}(e_1, e_1) \Vert^2 + \Vert \text{I\!I}(e_1, e_2) \Vert^2
\end{equation}
where $K$ is the Gauss curvature of $\Sigma$.  We also recall the \textit{Ricci equation}
\begin{align*}
\langle R^\perp_{12}e_\alpha, e_\beta \rangle & = \left\langle W_1(e_\beta), W_2(e_\alpha) \right\rangle - \left\langle W_1(e_\alpha), W_2(e_\beta)  \right\rangle
\end{align*}
where $3 \leq \alpha, \beta \leq 6$, and we are using the shorthand $R^\perp_{12} := R^\perp_{e_1, e_2}$ and $W_j := W_{e_j}$.  Expressing the second fundamental form as $\text{I\!I}(e_i, e_j) = h^\alpha_{ij}e_\alpha$, we have
\begin{align*}
W_1(e_\alpha) & = -h^\alpha_{11}e_1 - h^\alpha_{12}e_2 \\
W_2(e_\alpha) & = -h^\alpha_{12}e_1 + h^\alpha_{11}e_2,
\end{align*}
so that the Ricci equation reads
\begin{equation} \label{eq:RicciEq}
\langle R^\perp_{12}e_\alpha, e_\beta \rangle = 2 \left(h^\beta_{11}h^\alpha_{12} - h^\alpha_{11}h^\beta_{12}\right)\!.
\end{equation}

\subsubsection{First and Second Normal Bundles}

\indent \indent Since $u$ is a minimal surface, its second fundamental form $\text{I\!I}_p \colon \text{Sym}^2(T_p\Sigma) \to N_p\Sigma$ at $p \in \Sigma$ is determined by $\text{I\!I}_p(e_1, e_1)$ and $\text{I\!I}_p(e_1, e_2)$.  Therefore, the image of $\text{I\!I}_p$, called the \textit{first normal space}
\begin{align*}
\left.E_N\right|_p & := \left\{\text{I\!I}_p(X,Y) \in N_p\Sigma \colon X,Y \in T_p\Sigma\right\} = \text{span}(\text{I\!I}_p(e_1, e_1), \text{I\!I}_p(e_1, e_2)),
\end{align*}
is a vector space of dimension at most $2$.  Letting
\begin{align*}
\Sigma^\circ := \left\{p \in \Sigma \colon \dim(E_N|_p) = 2 \right\}
\end{align*}
we note that $\Sigma^\circ \subset \Sigma$ is an open set, and that $E_N := \bigcup_{p \in \Sigma^\circ} E_N|_p \to \Sigma^\circ$ is a rank $2$ vector bundle, called the \textit{first normal bundle}. \\
\indent For $p \in \Sigma^\circ$, let $E_B|_p$ denote the \textit{second normal space}, i.e., the orthogonal complement of $E_N|_p \subset N_p\Sigma$, so that there is an orthogonal splitting
$$N_p\Sigma = E_N|_p \oplus E_B|_p.$$
The rank $2$ vector bundle $E_B := \bigcup_{p \in \Sigma^\circ} E_B|_p \to \Sigma^\circ$ is called the \textit{second normal bundle}.  For a normal vector $\eta \in N\Sigma$, we write
\begin{equation} \label{eq:NBSplitting}
\eta = \eta^N + \eta^B
\end{equation}
for its decomposition into first normal and second normal components.  The \textit{third fundamental form} $\text{I\!I\!I} \colon \text{Sym}^3(T\Sigma) \to E_B$ is defined by $\text{I\!I\!I}(X,Y,Z) := [\nabla^\perp_X(\text{I\!I}(Y,Z))]^B$. 
It is a standard fact that $\text{I\!I\!I}$ is, in fact, symmetric in its arguments.

\subsection{Holomorphic Curves in $\mathbb{S}^6$} \label{sect:SU3-Structure}

\indent \indent Thus far, we have been regarding the round $\mathbb{S}^6$ simply as a Riemannian manifold.  We now equip it with extra data, namely its standard (nearly-K\"{a}hler) $\text{SU}(3)$-structure. To begin, let us consider the imaginary octonions $\text{Im}(\mathbb{O}) = \mathbb{R}^7$, equipped with the standard euclidean inner product $g_0$.  The imaginary octonions admit a well-known cross product $\times \colon \text{Im}(\mathbb{O}) \times \text{Im}(\mathbb{O}) \to \text{Im}(\mathbb{O})$ via 
$$x \times y := \textstyle\frac{1}{2}(xy - yx).$$
Using the metric $g_0$, the cross product can be recast as a $3$-form $\phi \in \Lambda^3(\mathbb{R}^7)^*$
$$\phi(x,y,z) := g_0(x \times y, z)$$
called the \textit{associative $3$-form}.  The associative $3$-form is a ($\text{G}_2$-invariant) calibration on $\mathbb{R}^7$, and its calibrated $3$-folds are called ``associative $3$-folds."  That is, an \textit{associative $3$-fold} is an immersed submanifold $N^3 \to \mathbb{R}^7$ that satisfies
\begin{equation*}
\left.\phi\right|_N = \text{vol}_N
\end{equation*}
where $\text{vol}_N$ is the volume form on $N^3$.  The study of associative $3$-folds is of fundamental importance to $\text{G}_2$-geometry \cite[$\S$12]{joyce2007}.  \\

\indent Returning to the round $6$-sphere, let us embed $\mathbb{S}^6 \subset \mathbb{R}^7 = \text{Im}(\mathbb{O})$ in the standard way.  For each $p \in \mathbb{S}^6$, we can use the cross product $\times$ to define a map
\begin{align*}
\widetilde{J}_p \colon T_p\mathbb{S}^6 & \to T_p\mathbb{S}^6 \\
\widetilde{J}_p(x) & = p \times x.
\end{align*}
The properties of $\times$ imply that each $(\widetilde{J}_p)^2 = -\text{Id}$.  The resulting bundle map $\widetilde{J} \colon T\mathbb{S}^6 \to T\mathbb{S}^6$ is the standard ($\text{G}_2$-invariant) almost-complex structure on the $6$-sphere.  One can check that each $\widetilde{J}_p \colon T_p\mathbb{S}^6 \to T_p\mathbb{S}^6$ is an isometry, and that the bilinear form on $\mathbb{S}^6$ given by
$$\widetilde{\Omega}(x,y) := \langle \widetilde{J}x, y \rangle$$
is skew-symmetric and non-degenerate.  In other words, the triple $(\langle \cdot, \cdot \rangle, \widetilde{J}, \widetilde{\Omega})$ is an \textit{almost-Hermitian} (or \textit{$\mathrm{U}(3)$-structure}) on $\mathbb{S}^6$.  We emphasize that $\widetilde{J}$ is not integrable, and that $\widetilde{\Omega}$ is not closed. \\
\indent Now, letting $\partial_r$ denote the radial vector field on $\mathbb{R}^7$, one can show that the complex $3$-form $\Upsilon \in \Omega^3(\mathbb{S}^6; \mathbb{C})$ given by
\begin{align*}
\Upsilon & := \left.\left(\partial_r \lrcorner(\ast \phi) + i\phi \right)\right|_{\mathbb{S}^6}
\end{align*}
is a $(3,0)$-form on $\mathbb{S}^6$ that satisfies
$$\frac{i}{8}\Upsilon \wedge \overline{\Upsilon} = \text{vol}_{\mathbb{S}^6}.$$
That is, the quadruple $(\langle \cdot, \cdot \rangle, \widetilde{J}, \widetilde{\Omega}, \Upsilon)$ is an \textit{$\mathrm{SU}(3)$-structure} on $\mathbb{S}^6$.  In fact, this $\text{SU}(3)$-structure satisfies the \textit{nearly-K\"{a}hler} equations $d\widetilde{\Omega} = 3\,\text{Im}(\Upsilon)$ and $d\,\text{Re}(\Upsilon) = 2\,\widetilde{\Omega} \wedge \widetilde{\Omega}$. 
The round $6$-sphere with this $\text{SU}(3)$-structure is the simplest example of a strict nearly-K\"{a}hler $6$-manifold. \\

\indent Now, the $\text{SU}(3)$-structure gives rise to distinguished classes of submanifolds of the $6$-sphere.  In particular, an immersed surface $u \colon \Sigma^2 \to \mathbb{S}^6$ is a \textit{holomorphic curve} if
$$\widetilde{J}(T_p\Sigma) = T_p\Sigma, \ \ \ \forall p \in \Sigma.$$
Holomorphic curves are, in fact, minimal surfaces.  One way to see this is to observe that holomorphic curves have extra symmetries in their second fundamental forms (see (\ref{eq:SecondFundSym}) in $\S$\ref{sect:FirstAdapt}), and these symmetries imply minimality.  Another way uses the following fundamental fact:

\begin{prop} Let $\Sigma^2 \subset \mathbb{S}^6$ be an immersed surface, and let $C(\Sigma) = \{rx \in \mathbb{R}^7 \colon r > 0, x \in \Sigma\}$ be its cone in $\mathbb{R}^7$.  Then $\Sigma$ is a holomorphic curve if and only if $C(\Sigma)$ is an associative $3$-fold.
\end{prop}

\indent So, as holomorphic curves are the links of associative cones, and since associative cones are homologically volume-minimizing, it follows that holomorphic curves are minimal surfaces.

\subsection{Holomorphic Bundles over Riemann Surfaces}

\indent \indent Let $E \to M$ be a complex vector bundle over a complex manifold $M$.  It is well-known \cite[$\S$1.3]{kobayashi14} that a holomorphic structure on $E$ is equivalent to a $\overline{\partial}$-operator, i.e., an operator
$$\overline{\partial} \colon \Gamma(E) \to \Omega^{0,1}(M) \otimes \Gamma(E)$$
satisfying both the relevant Leibniz rule and $\overline{\partial}^2 = 0$.  Given a complex vector bundle $E \to M$ equipped with both a connection $\nabla \colon \Gamma(E) \to \Omega^1(M; \mathbb{C}) \otimes \Gamma(E)$ and a holomorphic structure $\overline{\partial}$, we say that $\nabla$ and $\overline{\partial}$ are \textit{compatible} if $\nabla^{0,1} = \overline{\partial}$.

Note that if $E \to \Sigma$ is a complex vector bundle over a Riemann surface, then every connection $\nabla$ on $E$ has the property that $\nabla^{0,1}$ satisfies the Leibniz rule and squares to zero.  Said another way:

\begin{prop} \label{thm:Koszul-Malgrange} Let $E \to \Sigma$ be a complex vector bundle over a Riemann surface $\Sigma$.  For each connection $\nabla$, there exists a unique holomorphic structure on $E$ compatible with $\nabla$ (viz., $\overline{\partial} = \nabla^{0,1}$).
\end{prop}

\indent The holomorphic structure in Proposition \ref{thm:Koszul-Malgrange} is often called the \textit{Koszul-Malgrange holomorphic structure for $\nabla$}.  However, we shall occasionally abuse terminology and refer to $\nabla$ itself as the holomorphic structure.


\section{The Geometry of Holomorphic Curves in the $6$-Sphere}

\indent \indent In $\S$3.1 and $\S$3.2, we set up the moving frame for holomorphic curves in the $6$-sphere.  In $\S$3.3, we take a closer look at the class of \textit{null-torsion} holomorphic curves, the primary results being Proposition \ref{thm:HoloSections} and Proposition \ref{thm:GenusDegree}.  In $\S$3.4, we consider three different holomorphic structures on the normal bundle of a null-torsion holomorphic curve.

\subsection{Moving Frames for $\mathbb{S}^6$}


\indent \indent We begin by viewing $\mathbb{S}^6$ simply as an oriented Riemannian manifold (i.e., as a $6$-manifold with an $\text{SO}(6)$-structure).  Let $F_{\text{SO}(6)} \to \mathbb{S}^6$ denote the oriented orthonormal coframe bundle of $\mathbb{S}^6$.  Let $\omega \in \Omega^1(F_{\text{SO}(6)}; \mathbb{R}^6)$ denote the tautological $1$-form, and let $\psi \in \Omega^1(F_{\text{SO}(6)}; \mathfrak{so}(6))$ denote the Levi-Civita connection, so that we have
\begin{equation*}
d\omega = -\psi \wedge \omega.
\end{equation*}
So, if $(e_1, \ldots, e_6)$ is a local oriented orthonormal frame on an open set $U \subset \mathbb{S}^6$, then
\begin{equation} \label{eq:DFrames}
\overline{\nabla}e_i = -\psi_{ij} \otimes e_j
\end{equation}
where we are conflating (and will continue to conflate) the $1$-forms $\psi_{ij}$ on $F_{\text{SO}(6)}$ with their pullbacks $\sigma^*(\psi_{ij})$ on $U$ via the local section $\sigma \colon U \to F_{\text{SO}(6)}$ corresponding to $(e_1, \ldots, e_6)$.

\subsubsection{The $\text{SU}(3)$-Structure}


\indent \indent We now equip $\mathbb{S}^6$ with its standard $\text{SU}(3)$-structure $(\langle \cdot, \cdot \rangle, \widetilde{J}, \widetilde{\Omega}, \Upsilon)$, recalling $\S$\ref{sect:SU3-Structure}.  Let $\mathscr{P} := F_{\text{SU}(3)} \subset F_{\text{SO}(6)}$ denote the $\text{SU}(3)$-coframe bundle of $\mathbb{S}^6$.  There is a natural identification $\mathscr{P} \cong \text{G}_2$, but we will not use this fact explicitly.  Via the $\text{SU}(3)$-invariant splitting $\mathfrak{so}(6) = \mathfrak{su}(3) \oplus \mathbb{R}^6$ (orthogonal with respect to the Killing form), the restriction of the Levi-Civita connection to $\mathscr{P}$ decomposes as
\begin{equation} \label{eq:SU3Splitting}
\left.\psi\right|_{\mathscr{P}} = \widetilde{\gamma} + T(\omega),
\end{equation}
where $\widetilde{\gamma} \in \Omega^1(\mathscr{P}; \mathfrak{su}(3))$ is the natural $\text{SU}(3)$-connection and $T(\omega) \in \Omega^1(\mathscr{P}; \mathbb{R}^6)$ is the intrinsic torsion of the $\text{SU}(3)$-structure.  Thus, on $\mathscr{P}$, we have the first structure equations
\begin{equation*}
d\omega = -\widetilde{\gamma} \wedge \omega - T(\omega) \wedge \omega.
\end{equation*}
By writing the subspaces $\mathfrak{su}(3)$ and $\mathbb{R}^6$  of $\mathfrak{so}(6)$ in terms of explicit $6 \times 6$ matrices, we can express the connection matrix $\widetilde{\gamma}$ in the form
$$\widetilde{\gamma} = \left[\begin{array}{c c | c c | c c}
0 & -\beta_{11} & \alpha_{21} & -\beta_{21} & -\alpha_{31} & -\beta_{31} \\
\beta_{11} & 0 & \beta_{21} & \alpha_{21} & \beta_{31} & -\alpha_{31} \\ \hline
-\alpha_{21} & -\beta_{21} & 0 & -\beta_{22} & \alpha_{32} & -\beta_{32} \\ 
\beta_{21} & -\alpha_{21} & \beta_{22} & 0 & \beta_{32} & \alpha_{32} \\ \hline
\alpha_{31} & -\beta_{31} & -\alpha_{32} & -\beta_{32} & 0 & -\beta_{33} \\
\beta_{31} & \alpha_{31} & \beta_{32} & -\alpha_{32} & \beta_{33} & 0
\end{array}\right]$$
and calculate that the intrinsic torsion of $\mathbb{S}^6$ is
$$T(\omega) = \frac{1}{2}\left[\begin{array}{c c | c c | c c}
0 & 0 & \omega_5 & -\omega_6 & -\omega_3 & \omega_4 \\
0 & 0 & -\omega_6 & -\omega_5 & \omega_4 & \omega_3 \\ \hline
-\omega_5 & \omega_6 & 0 & 0 & \omega_1 & -\omega_2 \\ 
\omega_6 & \omega_5 & 0 & 0 & -\omega_2 & -\omega_1 \\ \hline
\omega_3 & -\omega_4 & -\omega_1 & \omega_2 & 0 & 0 \\
-\omega_4 & -\omega_3 & \omega_2 & \omega_1 & 0 & 0
\end{array}\right]\!.$$
Let $\overline{D} \colon \Gamma(T\mathbb{S}^6) \to \Omega^1(\mathbb{S}^6) \otimes \Gamma(T\mathbb{S}^6)$ denote the covariant derivative operator associated to the connection $\widetilde{\gamma}$.  If $(e_1, \ldots, e_6)$ is a local $\text{SU}(3)$-frame on $U \subset \mathbb{S}^6$, we have
\begin{equation} \label{eq:NablaFrames}
\overline{D} e_i = -\widetilde{\gamma}_{ij} \otimes e_j.
\end{equation}

\subsubsection{The $\text{SU}(3)$-Structure in Complex Notation}

\indent \indent It will often be convenient to have complex versions of the above equations.  To that end, let $\zeta \in \Omega^1(\mathscr{P}; \mathbb{C}^3)$ denote the complex tautological $1$-form, where:
\begin{align*}
\zeta_1 & = \omega_1 + i\omega_2 & \zeta_2 & = \omega_3 + i\omega_4 & \zeta_3 & = \omega_5 + i\omega_6
\end{align*}
Let $\gamma \in \Omega^1(\mathscr{P}; \mathfrak{su}(3))$ denote the natural $\text{SU}(3)$-connection on $\mathbb{S}^6$, regarded now as a complex $3 \times 3$ matrix (rather than a real $6 \times 6$ matrix).  In other words:
$$\left( \gamma_{ij} \right) = \begin{bmatrix}
\gamma_{11} & \gamma_{12} & \gamma_{13} \\
\gamma_{21} & \gamma_{22} & \gamma_{23} \\
\gamma_{31} & \gamma_{32} & \gamma_{33}
\end{bmatrix} = \begin{bmatrix}
i\beta_{11} & \alpha_{21} + i\beta_{21} & -\alpha_{31} + i\beta_{31} \\
-\alpha_{21} + i\beta_{21} & i\beta_{22} & \alpha_{32} + i\beta_{32} \\
\alpha_{31} + i\beta_{31} & -\alpha_{32} + i\beta_{32} & i\beta_{33} 
\end{bmatrix}\!.$$
In this notation, the \textit{first} and \textit{second structure equations} of $\mathbb{S}^6$ are \cite{bryant82}, \cite{bryant2006geometry}
\begin{align}
d\zeta_i & = -\gamma_{i\ell} \wedge \zeta_\ell + \overline{\zeta}_j \wedge \overline{\zeta}_k  \label{eq:StrEq1}  \\
d\gamma_{ij} & = \textstyle -\gamma_{ik} \wedge \gamma_{kj} + \frac{3}{4}\zeta_i \wedge \overline{\zeta}_j - \frac{1}{4} \delta_{ij}\, \zeta_\ell \wedge \overline{\zeta}_\ell  \label{eq:StrEq2}
\end{align}
where $(i,j,k)$ in the first structure equation is an even permutation of $(1,2,3)$. \\
\indent Extend both $\overline{\nabla}$ and $\overline{D}$ by $\mathbb{C}$-linearity to operators $\Gamma(T\mathbb{S}^6 \otimes_{\mathbb{R}} \mathbb{C}) \to \Gamma(T\mathbb{S}^6\otimes_{\mathbb{R}} \mathbb{C}) \otimes \Omega^1(\mathbb{S}^6; \mathbb{C})$.  In terms of a local $\text{SU}(3)$-frame $(e_1, \ldots, e_6)$ for $T\mathbb{S}^6$, we let
\begin{align*}
f_1 & = \frac{1}{2}(e_1 - ie_2) & f_2 & = \frac{1}{2}(e_3 - ie_4) & f_3 & = \frac{1}{2}(e_5 - ie_6) \\
\overline{f}_1 & = \frac{1}{2}(e_1 + ie_2) & \overline{f}_2 & = \frac{1}{2}(e_3 + ie_4) & \overline{f}_3 & = \frac{1}{2}(e_5 + ie_6).
\end{align*}
Note that $(f_1, f_2, f_3)$ is a local $\text{SU}(3)$-frame for $T^{1,0}\mathbb{S}^6$, while $(\overline{f}_1, \overline{f}_2, \overline{f}_3)$ is a local $\text{SU}(3)$-frame for $T^{0,1}\mathbb{S}^6$.  A calculation shows that
\begin{align}
\overline{\nabla}f_1 & = \textstyle \overline{D} f_1 + \frac{1}{2}( \zeta_2 \otimes \overline{f}_3 - \zeta_3 \otimes \overline{f}_2) \nonumber \\
\overline{\nabla}f_2 & = \textstyle \overline{D} f_2 + \frac{1}{2}( \zeta_3 \otimes \overline{f}_1 - \zeta_1 \otimes \overline{f}_3) \label{eq:CLinD} \\
\overline{\nabla}f_3 & = \textstyle \overline{D} f_3 + \frac{1}{2}( \zeta_1 \otimes \overline{f}_2 - \zeta_2 \otimes \overline{f}_1) \nonumber
\end{align}
and
\begin{equation}  \label{eq:NablaFramesCplx}
\overline{D} f_i = \gamma_{ji}\otimes f_j 
\end{equation}
where we underscore that $(\gamma_{ji}) = (\gamma_{ij})^T$. 

\subsection{Moving Frames for Holomorphic Curves in $\mathbb{S}^6$} \label{sec:MovingFramesHoloCurves}

\indent \indent We now turn our attention to holomorphic curves $u \colon \Sigma^2 \to \mathbb{S}^6$, always assuming for simplicity that $u$ is an (unramified) immersion.  In this section, we recall Bryant's ``holomorphic Frenet frame" for $u$, which will be central to our calculations.  Our discussion is essentially a self-contained summary of \cite[$\S$4]{bryant82}, though we have changed notation in several places.  Preparation of this section was aided by clarifying discussions in \cite{hashimoto2000} and \cite{lotay2011ruled}.

\indent Before getting started, let us give a brief overview of the various complex vector bundles over $\Sigma$ that we will need.  First, consider the complex rank $3$ bundle $u^*(T^{1,0}\mathbb{S}^6) \to \Sigma$ of $(1,0)$-vectors along $u(\Sigma)$.  Next, we let
$$L_T := T^{1,0}\Sigma \subset u^*(T^{1,0}\mathbb{S}^6)$$
and define the complex rank $2$ bundle
$$Q_{NB} := u^*(T^{1,0}\mathbb{S}^6)/L_T.$$
For $v \in u^*(T^{1,0}\mathbb{S}^6)$, we let $(v) \in Q_{NB}$ denote its projection to the quotient.

In the sequel, we will define a certain complex line subbundle $L_N \subset Q_{NB}$, from which we will set $L_B := Q_{NB}/L_N$.  As above, for $(v) \in Q_{NB}$, we let $(\!(v)\!) \in L_B$ denote its projection to $L_B$.  In summary, we have a diagram:
$$\begin{tikzcd}
L_T \arrow[r] & u^*(T^{1,0}\mathbb{S}^6) \arrow[d, "(\cdot)"] \\
L_N \arrow[r] & Q_{NB}  \arrow[d, "(\!(\cdot)\!)"] \\
                    & L_B
\end{tikzcd}$$
All complex vector bundles under consideration are assumed to be endowed with their obvious Hermitian metrics.  As Hermitian vector bundles, we will have isomorphisms
\begin{align}
u^*(T^{1,0}\mathbb{S}^6) & \simeq L_T \oplus L_N \oplus L_B \label{eq:CBundIso1} \\
Q_{NB} & \simeq L_N \oplus L_B. \label{eq:CBundIso2}
\end{align}
We will shortly equip all of these bundles with holomorphic structures, cautioning that the isomorphisms (\ref{eq:CBundIso1}) and (\ref{eq:CBundIso2}) generally will not hold in the holomorphic category.

\subsubsection{Holomorphic Structures} \label{sect:FirstAdapt}

\indent \indent To begin, recall the (complexified) $\text{SU}(3)$-connection $\overline{D}$ on $T\mathbb{S}^6 \otimes_{\mathbb{R}} \mathbb{C}$.  By restriction and pullback, we get an induced connection (still denoted $\overline{D}$) on $u^*(T^{1,0}\mathbb{S}^6) \to \Sigma$.  We endow $u^*(T^{1,0}\mathbb{S}^6)$ with the Koszul-Malgrange holomorphic structure for $\overline{D}$.  Since $u$ is an immersion, the complex line bundle $L_T := T^{1,0}\Sigma \subset u^*(T^{1,0}\mathbb{S}^6)$ is a holomorphic line subbundle, and we equip the quotient bundle $Q_{NB} := u^*(T^{1,0}\mathbb{S}^6)/L_T$ with the induced holomorphic structure.  These structures in place, we now make two frame adaptations.

\subsubsection{First Adaptation} \label{sect:FirstAdapt}

\indent \indent Let $(f_1, f_2, f_3)$ be an $\text{SU}(3)$-frame for $u^*(T^{1,0}\mathbb{S}^6)$.  For our first adaptation, we consider those frames for which
$$f_1 \in L_T = T^{1,0}\Sigma.$$
The set of such frames comprises a $\text{U}(2)$-subbundle $\mathscr{F}_1 \subset u^*\mathscr{P}$ over $\Sigma$, and we will refer to such $(f_1, f_2, f_3)$ as being \textit{$\mathrm{U}(2)$-adapted}.  On $\mathscr{F}_1$, we have that $\zeta_2 = \zeta_3 = 0$.  Differentiating these equations and applying Cartan's Lemma shows that there exist functions $\kappa, \mu \colon \mathscr{F}_1 \to \mathbb{C}$ for which
\begin{align*}
\gamma_{21} & = \kappa\zeta_1 & \gamma_{31} & = \mu\zeta_1.
\end{align*}
Writing $\kappa = \kappa_1 + i\kappa_2$ and $\mu = \mu_1 + i\mu_2$, where $\kappa_1, \kappa_2, \mu_1, \mu_2 \colon \mathscr{F}_1 \to \mathbb{R}$, a calculation shows that the second fundamental form may be expressed as
\begin{align} 
\text{I\!I}(e_1, e_1) & = \kappa_1e_3 + \kappa_2 e_4 + \mu_1 e_5 - \mu_2e_6 \nonumber \\
\text{I\!I}(e_1, e_2) & = -\kappa_2 e_3 + \kappa_1 e_4 + \mu_2 e_6 + \mu_1e_6 \label{eq:SecondFundSym} \\
\text{I\!I}(e_2, e_2) & = -\text{I\!I}(e_1, e_1). \nonumber
\end{align}
In particular, we observe that holomorphic curves are minimal surfaces.  Equation (\ref{eq:SecondFundSym}) also shows that the functions $\kappa, \mu$ are essentially equivalent to the second fundamental form. \\
\indent Using $\text{U}(2)$-adapted frames, we can understand the holomorphic structures on $L_T$ and $Q_{NB}$ more explicitly.  That is, the Chern connection of $L_T$ is given by
$$D^{L_T}f_1 = \gamma_{11} \otimes f_1$$
while the Chern connection of $Q_{NB}$ is given by
\begin{align*}
D^{Q_{NB}}(f_2) & = \gamma_{22} \otimes (f_2) + \gamma_{32} \otimes (f_3) \\
D^{Q_{NB}}(f_3) & = \gamma_{23} \otimes (f_2) + \gamma_{33} \otimes (f_3).
\end{align*}
\indent We now recast the second fundamental form as a holomorphic section.  In \cite[Lemma 4.3]{bryant82}, it is shown that
\begin{align*}
\Phi_{\text{I\!I}} & \in \Gamma(  L_T^* \otimes L_T^* \otimes Q_{NB}) \\ 
\Phi_{\text{I\!I}} & = \kappa\zeta_1 \otimes f^\vee_1 \otimes  (f_2) +  \mu\zeta_1 \otimes f^\vee_1  \otimes (f_3)
\end{align*}
is a well-defined (frame-independent) holomorphic section, where $f^\vee_1 \colon \mathscr{F}_1 \to L_T^*$ is the dual of $f_1$.  It is remarked in \cite[Lemma 4.4]{bryant82} that $\Phi_{\text{I\!I}} = 0$ if and only if $u$ is the totally geodesic $\mathbb{S}^2$.  On the other hand, if $u$ is not the totally geodesic $\mathbb{S}^2$, then the zeros of $\Phi_{\text{I\!I}}$ are isolated, hence finite (since $\Sigma$ is compact).  To streamline further discussion, we enact the following: \\

\noindent \textbf{Convention:} From now on, we assume that $u$ is not totally-geodesic. \\

\indent It is convenient to regard $\Phi_{\text{I\!I}}$ as a holomorphic section of $\text{Hom}( L_T \otimes L_T; Q_{NB})$.  
Thus, there is a holomorphic line subbundle, called $L_N \subset Q_{NB}$, such that
$$\Phi_{\text{I\!I}} \in H^0( \text{Hom}( L_T \otimes L_T; L_N) ).$$
To be more explicit, let $F$ denote the (effective) divisor of the holomorphic section $\Phi_{\text{I\!I}}$, i.e.,
$$F = \sum_{p \in \Sigma \colon \Phi_{\text{I\!I}}(p) = 0} \text{ord}_p(\Phi_{\text{I\!I}}) \cdot p,$$
and let $\mathcal{O}_F \to \Sigma$ be the corresponding holomorphic line bundle.  Viewing $\Phi_{\text{I\!I}} \in H^0( (L_T \otimes L_T)^* \otimes L_N )$, it follows that
$$L_N = \mathcal{O}_F \otimes L_T \otimes L_T.$$
Finally, we let $L_B := Q_{NB}/L_N$ and equip $L_B$ with the induced holomorphic structure.

\subsubsection{Second Adaptation}

\indent \indent For our second adaptation, we consider the $\text{U}(2)$-adapted frames $(f_1, f_2, f_3)$ for which
$$(f_2) \in L_N.$$
This adaptation defines a $T^2$-subbundle $\mathscr{F}_2 \subset \mathscr{F}_1 \subset u^*\mathscr{P}$ over $\Sigma$, and we refer to such frames as \textit{$T^2$-adapted}.  On $\mathscr{F}_2$, we have that $\gamma_{31} = 0$, so that $\mu = 0$.  Differentiating $\gamma_{31} = 0$ shows that $\gamma_{32}$ is a semibasic $(1,0)$-form, and hence
$$\gamma_{32} = \tau\zeta_1$$
for some function $\tau \colon \mathscr{F}_2 \to \mathbb{C}$.  In summary, if $(f_1, f_2, f_3)$ is a $T^2$-adapted frame, the equations (\ref{eq:NablaFramesCplx}) now read:  
\begin{equation} \label{eq:HoloFrenetFrame}
\overline{D}\begin{bmatrix}
f_1 \\
f_2 \\
f_3 \end{bmatrix} =
\begin{pmatrix}
\gamma_{11} & \kappa\zeta_1 & 0 \\
-\overline{\kappa}\overline{\zeta}_1  & \gamma_{22} &  \tau\zeta_1  \\
0 & -\overline{\tau}\overline{\zeta_1} & \gamma_{33}
\end{pmatrix} \otimes
\begin{bmatrix}
f_1 \\
f_2 \\
f_3 \end{bmatrix}\!.
\end{equation}
These are the \textit{holomorphic Frenet equations} for the holomorphic curve $u \colon \Sigma^2 \to \mathbb{S}^6$. \\
\indent Using $T^2$-adapted frames, we can understand the holomorphic structures on $L_N$ and $L_B$ more explicitly.  That is, the Chern connection of $L_N$ is given by
$$D^{L_N}(f_2) = \gamma_{22} \otimes (f_2)$$
and that of $L_B$ by
$$D^{L_B}(\!(f_3)\!) = \gamma_{33} \otimes (\!(f_3)\!).$$
\indent Now, by analogy with the familiar Frenet frame for curves in $\mathbb{R}^3$, one might be inclined to call $\tau$ the ``holomorphic torsion" of $u \colon \Sigma^2 \to \mathbb{S}^6$, but for the fact that $\tau$ depends on the choice of $T^2$-frame $(f_1, f_2, f_3)$.  However, the ``null-torsion" condition $\tau = 0$ turns out to be independent of frame.  Indeed, Bryant shows \cite[Lemma 4.5]{bryant82} that
\begin{align*}
\Phi_{\text{I\!I\!I}} & \in \Gamma( L_T^* \otimes L_N ^* \otimes L_B) \\ 
\Phi_{\text{I\!I\!I}} & =  \tau\zeta_1 \otimes (f^\vee_2) \otimes (\!( f_3 )\!)
\end{align*}
is a well-defined (frame-independent) holomorphic section.  The section $\Phi_{\text{I\!I\!I}}$ partitions the collection of (non-totally-geodesic) holomorphic curves into three classes: 
\begin{enumerate}
\item $\Phi_{\text{I\!I\!I}} = 0$ identically.
\item The zero set of $\Phi_{\text{I\!I\!I}}$ is finite and non-empty.
\item $\Phi_{\text{I\!I\!I}}$ is nowhere-vanishing.
\end{enumerate}
The generic situation is (2), and relatively little is known about this case.  As Bryant remarks \cite[p. 225]{bryant82}, the condition (3) is quite strong, implying a stringent relation on the line bundles $L_T, L_N, L_B$. 

\indent Holomorphic curves of type (1) are said to be \textit{null-torsion}, and are the focus of this work.  Note that every holomorphic curve of genus zero is null-torsion \cite[Theorem 4.6]{bryant82} or the totally-geodesic $2$-sphere.  It is shown in \cite{bolton94} that every null-torsion holomorphic curve is linearly full in $\mathbb{S}^6$ (i.e., is not contained in a totally-geodesic $\mathbb{S}^5$), implying that even the simplest null-torsion curves cannot be reduced to the study of minimal Legendrians in $\mathbb{S}^5$.

\indent It is a remarkable fact \cite[Theorem 4.10]{bryant82} that every compact Riemann surface admits a conformal branched immersion into $\mathbb{S}^6$ as a null-torsion holomorphic curve.  In his 1999 Ph.D. thesis \cite{rowland}, Rowland extended this result, showing that, in fact, every compact Riemann surface may be conformally \textit{embedded} as a null-torsion holomorphic curve in $\mathbb{S}^6$.

\begin{rmk} As of this writing, it is an open question whether every \textit{open} Riemann surface can be conformally embedded as a null-torsion holomorphic curve in $\mathbb{S}^6$.  It seems to the author that the techniques of \cite{alarcon19} may yield a positive solution to this problem.
\end{rmk}

\subsection{Null-Torsion Holomorphic Curves}

\indent \indent We now examine null-torsion holomorphic curves more closely.  In Proposition \ref{thm:NullTorsionEquiv}, we give two holomorphic interpretations of the null-torsion condition.  Then, in Proposition \ref{thm:c1LineBundles}, we will see how the null-torsion condition constrains the topologies of the line bundles $L_T, L_N, L_B$.  Using this, together with Riemann-Roch, we will calculate (Proposition \ref{thm:HoloSections}) the number of independent holomorphic sections of $L_N$ and $L_B^*$.

\subsubsection{Holomorphic Interpretations of Null-Torsion}

\indent \indent Let $u \colon \Sigma^2 \to \mathbb{S}^6$ holomorphic curve, and regard $\mathbb{S}^6 \subset \mathbb{R}^7$ as the usual unit sphere.  Its \textit{binormal Gauss map} is
\begin{align*}
b_u \colon \Sigma^2 & \to \text{Gr}_2^+(\mathbb{R}^7) \\
b_u(p) & = e_5 \wedge e_6 
\end{align*}
where $(e_1, \ldots, e_6)$ is a $T^2$-frame at $p \in \Sigma$.  One can check that $b_u$ is well-defined, independent of frame.  Now, consider the map
\begin{align*}
\text{Gr}_2^+(\mathbb{R}^7) & \to \mathbb{P}(\mathbb{C}^7) = \mathbb{CP}^6 \\
x \wedge y & \mapsto \text{span}_{\mathbb{C}}(x - iy)
\end{align*}
where $\{x,y\}$ is orthonormal.  This map is well-defined (independent of basis), injective, and its image is the complex hypersurface $\Lambda = \{[z] \in \mathbb{CP}^6 \colon z_1^2 + \cdots + z_7^2 = 0\}$.  Consequently, we may identify $\text{Gr}_2^+(\mathbb{R}^7) \simeq \Lambda \subset \mathbb{CP}^6$.

\begin{prop} \label{thm:NullTorsionEquiv} Let $u \colon \Sigma^2 \to \mathbb{S}^6$ holomorphic curve.  The following are equivalent: \\
\indent (i) $u$ is null-torsion. \\
\indent (ii) The binormal Gauss map $b_u \colon \Sigma^2 \to \Lambda$ is holomorphic. \\
\indent (iii) There is a holomorphic splitting $Q_{NB} \cong L_N \oplus L_B$.
\end{prop} 

\begin{proof} The equivalence of (i) and (ii) is \cite[Theorem 4.7]{bryant82}.  For the equivalence of (i) and (iii), we simply observe that $\tau$ is essentially the second fundamental form (in the sense of complex geometry, cf. \cite[Chap. V: $\S$14]{Demailly97} or \cite[$\S$1.6]{kobayashi14}) of the holomorphic Hermitian subbundle $L_N \subset Q_{NB}$.  . \end{proof}

\begin{cor} \label{thm:WirtDeg} If $u \colon \Sigma^2 \to \mathbb{S}^6$ is a null-torsion holomorphic curve, then its area $A = 4\pi d$, where $d$ is the degree of the binormal Gauss map.
\end{cor}

\subsubsection{Topological Consequences}

\indent \indent We now consider the topologies of the bundles $L_T, L_N, L_B$. Now, as Bryant points out \cite[p. 224]{bryant82}, the $\text{SU}(3)$-structure on $\mathbb{S}^6$ yields a holomorphic, metric isomorphism
$$\Lambda^3(T^{1,0}\mathbb{S}^6) \cong \underline{\mathbb{C}}$$
where $\underline{\mathbb{C}}$ is the trivial line bundle.  Consequently, there is an isomorphism of holomorphic line bundles
$$L_T \otimes L_N \otimes L_B \cong \underline{\mathbb{C}}$$
In particular, it follows that:
\begin{equation} \label{eq:c1sumzero}
c_1(L_T) + c_1(L_N) + c_1(L_B) = 0.
\end{equation}
 \\
\indent Here is an equivalent way to see this.  Since $\gamma$ is valued in $\mathfrak{su}(3)$, we have 
$$\gamma_{11} + \gamma_{22} + \gamma_{33} = 0.$$
Note that $\gamma_{11}$, $\gamma_{22}$, $\gamma_{33}$ are, respectively, the connection forms of the Chern connections on $L_T, L_N, L_B$.  From the structure equations (\ref{eq:StrEq2}), we may compute that their curvature $(1,1)$-forms $F_T, F_N, F_B$ are given by
\begin{align*}
F_T = i\,d\gamma_{11} & = \left(1 - 2|\kappa|^2 \right) \text{vol} =: K_T\,\text{vol} \\
F_N = i\,d\gamma_{22} & = \left( 2|\kappa|^2 - 2 |\tau|^2 - \frac{1}{2} \right) \text{vol} =: K_N\,\text{vol} \\
F_B = i\,d\gamma_{33} & = \left( 2 |\tau|^2 - \frac{1}{2} \right) \text{vol} =: K_B\,\text{vol}
\end{align*}
where $\text{vol} = \omega_1 \wedge \omega_2$ is the volume form on $\Sigma$, and where $K_T$, $K_N$, $K_B$ are defined by these equations.  Note that $K_T = K$ is simply the Gauss curvature of $\Sigma$.  Thus, we have
$$F_T + F_N + F_B = 0,$$
and hence
\begin{align*}
\frac{1}{2\pi} \int_\Sigma K_T \,\text{vol} + \frac{1}{2\pi} \int_\Sigma K_N \,\text{vol} + \frac{1}{2\pi} \int_\Sigma K_B \,\text{vol} = 0,
\end{align*}
which gives (\ref{eq:c1sumzero}) by Chern-Weil theory.  In the null-torsion case, the formulas for $K_N$ and $K_B$ simplify, yielding:

\begin{prop} \label{thm:c1LineBundles} If $u \colon \Sigma^2 \to \mathbb{S}^6$ is a null-torsion holomorphic curve of area $A = 4\pi d$, then
\begin{align*}
c_1(L_T) & = \chi(\Sigma) \\
c_1(L_N) & = -\chi(\Sigma) + d \\
c_1(L_B) & = -d.
\end{align*}
Moreover, counting with multiplicity, there are exactly $d - 3\chi(\Sigma)$ points $p \in \Sigma$ at which $\Phi_{\mathrm{I\!I}}(p) = 0$.
\end{prop}

\begin{proof} It is a standard fact that $c_1(L_T) = c_1(T^{1,0}\Sigma) = \chi(\Sigma)$.  Since $u$ is null-torsion, we have $\tau = 0$, so $K_B = -\frac{1}{2}$, so:
$$c_1(L_B) = \frac{1}{2\pi} \int_\Sigma K_B\,\text{vol} = -\frac{A}{4\pi}.$$
Note that this gives another proof of the fact that null-torsion holomorphic curves have area equal to $4\pi$ times a positive integer.  From Corollary \ref{thm:WirtDeg}, we know that $A = 4\pi d$, where $d$ is the degree of the binormal lift.  So, we obtain:
$$c_1(L_B) = -d.$$
The last claim is simply that
$$\sum_{p \in \Sigma \colon \Phi_{\text{I\!I}}(p) = 0} \text{ord}_p(\Phi_{\text{I\!I}}) = \deg(F) = c_1(\mathcal{O}_F) = c_1(L_T^* \otimes L_T^* \otimes L_N) = d - 3\chi(\Sigma).$$
\end{proof}

\subsubsection{Holomorphic Consequences}

\indent \indent We now invoke Riemann-Roch to understand the spaces of holomorphic sections of $L_N$ and $L_B^*$.

\begin{prop} \label{thm:HoloSections} Suppose $u \colon \Sigma^2 \to \mathbb{S}^6$ is a null-torsion holomorphic curve of area $A = 4\pi d$.  Then: \\
\indent (a) We have:
\begin{align*}
h^0(L_N) & = d - \frac{1}{2}\chi(\Sigma) \\
h^0(L_B^*) & = d + \frac{1}{2}\chi(\Sigma) + h^0(L_B \otimes K_\Sigma).
\end{align*}
\indent (b) If $d > 2g-2$, then $h^0(L_B \otimes K_\Sigma) = 0$.
\end{prop}

\begin{proof} (a) To begin, observe that
$$\deg(L_N^* \otimes K_\Sigma) = -c_1(L_N) - \chi(\Sigma) = -d < 0.$$
Since the line bundle $L_N^* \otimes K_\Sigma$ is negative, it has no non-trivial holomorphic sections.  Therefore, by Riemann-Roch, Serre Duality, and Proposition \ref{thm:c1LineBundles}, we obtain:
\begin{align*}
h^0(L_N) & = h^0(L_N^* \otimes K_\Sigma) + c_1(L_N) + \frac{1}{2}\chi(\Sigma) \\
& = 0 - \chi(\Sigma) + d + \frac{1}{2}\chi(\Sigma).
\end{align*}
Similarly, we have:
\begin{align*}
h^0(L_B^*) & = h^0(L_B \otimes K_\Sigma) + c_1(L_B^*) + \frac{1}{2}\chi(\Sigma) \\
& = d + \frac{1}{2}\chi(\Sigma) + h^0(L_B \otimes K_\Sigma).
\end{align*}
\indent (b) Letting $\mathscr{L} := L_B \otimes K_\Sigma$, we compute
$$\deg(\mathscr{L}) = c_1(L_B) + c_1(K_\Sigma) = -d + (2g -2).$$
Therefore, if $d > 2g-2$, then $\deg(\mathscr{L}) < 0 $, so $\mathscr{L}$ has no non-trivial holomorphic sections.
\end{proof}

\indent In light of Proposition \ref{thm:HoloSections}(b), it is of interest to know when the hypothesis ``$d > 2g - 2$" might be automatically satisfied.  To that end, we recall the following facts from complex algebraic geometry.  See \cite[$\S$2.3, page 253]{griffiths78} for a proof.

\begin{lem} \label{thm:AlgGeo} Let $\widetilde{u} \colon \Sigma^2 \to \mathbb{CP}^6$ be a non-degenerate complex curve of genus $g$ and degree $d$. \\
\indent (a) We have $d \geq 6$. \\
\indent (b) If $d = 6$, then $\widetilde{u}$ is the rational normal curve, which has genus $g = 0$. \\
\indent (c) If $7 \leq d \leq 11$, then $g \leq d - 6$.
\end{lem}
\begin{prop} \label{thm:GenusDegree} Let $u \colon \Sigma^2 \to \mathbb{S}^6$ be a null-torsion holomorphic curve, where $\Sigma$ is a closed surface of genus $g$ and area $A = 4\pi d$.  If $g \leq 6$, then $d > 2g-2$. 
\end{prop}

\begin{proof} If $u$ is totally-geodesic, the claim is trivial, so assume otherwise.  Then its binormal Gauss map $b_u \colon \Sigma^2 \to \Lambda \subset \mathbb{CP}^6$ is non-degenerate, so $d \geq 6$.  If $d = 6$, then by Lemma \ref{thm:AlgGeo}(b), we have $g = 0$, fulfilling the bound $d > 2g-2$.  Thus, we may assume that $d \geq 7$ for the remainder of this proof. \\
\indent If $g \leq 4$, then $2g - 2 \leq 6 < 7 \leq d$, so the bound is fulfilled in this case.  Suppose now that $g = 5$.  If we had $d \leq 2g-2 = 8$, then Lemma \ref{thm:AlgGeo}(c) would imply that $g \leq 2$, which is absurd.  Analogous reasoning holds for $g = 6$.
\end{proof}

\subsection{Structures on the Normal Bundle} \label{sec:StructuresOnNormal}

\indent \indent Let $u \colon \Sigma^2 \to \mathbb{S}^6$ be a holomorphic curve (which may or may not be null-torsion).  With respect to the splitting $u^*(T\mathbb{S}^6) = T\Sigma \oplus N\Sigma$, let $\nabla^\top$, $\nabla^\perp$ and $D^\top$, $D^\perp$ denote the tangential and normal connections for $\overline{\nabla}$ and $\overline{D}$.  In this section, we equip the normal bundle $N\Sigma$ with various holomorphic structures and compare their properties.  Since this is an important but perhaps technical point, we now provide a brief overview of this section. \\
\indent Thus far, we have been focused on the $\text{G}_2$-invariant almost-complex structure $\widetilde{J}$ on $\mathbb{S}^6$.  By restriction, we get a complex structure, also called $\widetilde{J}$, on $N\Sigma$.  Decomposing the complexification $N\Sigma \otimes_{\mathbb{R}} \mathbb{C}$ into $\widetilde{J}$-eigenbundles
$$N\Sigma \otimes_{\mathbb{R}} \mathbb{C} = \widetilde{N}^{1,0} \oplus \widetilde{N}^{0,1}$$
we will ask how the $\mathbb{C}$-linear extensions of $\nabla^\perp$ and $D^\perp$ relate to this decomposition.  The upshot is that the complex bundle $\widetilde{N}^{1,0}$ can always be endowed with a holomorphic structure $\overline{\partial}^{\text{SU}}$ compatible with $D^\perp$.  As shown by Lotay \cite{lotay11asymp}, this holomorphic structure plays a key role in the deformation theory of associative $3$-folds in $\mathbb{R}^7$ asymptotic to the cone on $\Sigma$. \\
\indent However, for our study of the second variation of area, we will need to consider a \textit{different} complex structure on $N\Sigma = E_N \oplus E_B$, which we call $\widehat{J}$.  This alternate complex structure $\widehat{J}$ will agree with $\widetilde{J}$ on $E_N$, but differ on $E_B$.  Again, we will decompose $N\Sigma \otimes_{\mathbb{R}} \mathbb{C}$ into $\widehat{J}$-eigenbundles
$$N\Sigma \otimes_{\mathbb{R}} \mathbb{C} = \widehat{N}^{1,0} \oplus \widehat{N}^{0,1}$$
and consider how the $\mathbb{C}$-linear extensions of $\nabla^\perp$ and $D^\perp$ interact with this decomposition.  The result is that in the \textit{null-torsion case}, the complex bundle $\widehat{N}^{1,0}$ can be equipped with two different holomorphic structures: one called $\overline{\partial}^\nabla$ that is compatible with $\nabla^\perp$, and another called  $\overline{\partial}^{D}$ that is compatible with $D^\perp$.  Comparing these two structures on $\widehat{N}^{1,0}$ is the idea behind Theorem 1.1.

\subsubsection{The Induced Complex Structure} \label{sec:J1}

\indent \indent Let $u \colon \Sigma \to \mathbb{S}^6$ be a holomorphic curve.  It is easy to check that the covariant derivative operators $\nabla^\perp$ and $D^\perp$ on the normal bundle $N\Sigma$ interact with the complex structure $\widetilde{J}$ as follows: 

\begin{prop} \label{prop:JTilde} 
The complex structure $\widetilde{J}$ is not $\nabla^\perp$-parallel, but is $D^\perp$-parallel (i.e., $D^\perp \widetilde{J} = 0$).
\end{prop}

\indent We now complexify $N\Sigma$ and extend $\widetilde{J}$ by $\mathbb{C}$-linearity.  In the usual way, we have a decomposition of $N\Sigma \otimes_{\mathbb{R}} \mathbb{C}$ into $\widetilde{J}$-eigenbundles, say
$$N\Sigma \otimes_{\mathbb{R}} \mathbb{C} = \widetilde{N}^{1,0} \oplus \widetilde{N}^{0,1}$$
where, for example, $\widetilde{N}^{1,0} = \{\xi \in N\Sigma \otimes_{\mathbb{R}} \mathbb{C} \colon \widetilde{J}\xi = i\xi \} = \textstyle \{ \frac{1}{2}(\eta - i\widetilde{J}\eta) \colon \eta \in N\Sigma \}$.
As complex vector bundles, there is an isomorphism $(N\Sigma, \widetilde{J}) \simeq (\widetilde{N}^{1,0}, i)$
via $\eta \mapsto \frac{1}{2}(\eta - i\widetilde{J}\eta)$ with inverse $\xi \mapsto \frac{1}{2}(\xi + \overline{\xi})$. There is also a well-defined isomorphism of complex vector bundles
\begin{align}
(\widetilde{N}^{1,0}, i) & \to Q_{NB} \label{eq:NQIso} \\
v^\perp & \mapsto (v) \nonumber
\end{align}
where we decompose $v \in u^*(T^{1,0}\mathbb{S}^6)$ as $v = v^\top + v^\perp$ with tangential part $v^\top \in T^{1,0}\Sigma$ and normal part $v^\perp \in \widetilde{N}^{1,0}\Sigma$.  In terms of a local $T^2$-frame $(e_1, \ldots, e_6)$, the isomorphisms $(N\Sigma, \widetilde{J}) \simeq (\widetilde{N}^{1,0}, i) \simeq Q_{NB} \simeq L_N \oplus L_B$ are simply:
\begin{align*}
e_3 & \mapsto f_2 \mapsto (f_2) \mapsto (f_2) \oplus 0 & e_5 & \mapsto f_3 \mapsto (f_3) \mapsto 0 \oplus (\!(f_3)\!) \\
e_4 & \mapsto if_2 \mapsto (if_2) \mapsto (if_2) \oplus 0 & e_6 & \mapsto if_3 \mapsto (if_3) \mapsto 0 \oplus (\!(if_3)\!).
\end{align*}

\indent Now, we have already equipped $Q_{NB}$, $L_N$, and $L_B$ with holomorphic structures, and we would like to endow $\widetilde{N}^{1,0}$ with a holomorphic structure as well.  There is an obvious way to do this: we can use the isomorphism (\ref{eq:NQIso}) to pull back the holomorphic structure on $Q_{NB}$ to $\widetilde{N}^{1,0}$.  Unwinding the definitions shows that this is precisely the Koszul-Malgrange structure for the $\mathbb{C}$-linear extension of $D^\perp$ on $\widetilde{N}^{1,0}$.

\indent Let us illustrate this holomorphic structure in terms of a local $T^2$-frame $(f_1, f_2, f_3)$.  To begin, extend both $\nabla^\perp$ and $D^\perp$ by $\mathbb{C}$-linearity to operators
$$\nabla^\perp, D^\perp \colon \Gamma(N\Sigma \otimes_{\mathbb{R}} \mathbb{C}) \to \Omega^1(\Sigma; \mathbb{C}) \otimes \Gamma(N\Sigma \otimes_{\mathbb{R}} \mathbb{C}).$$
From (\ref{eq:CLinD}) and (\ref{eq:HoloFrenetFrame}), we see that
\begin{align*}
\nabla^\perp f_2 & = \textstyle \gamma_{22} \otimes f_2 + \tau\zeta_1 \otimes f_3 - \frac{1}{2}\zeta_1 \otimes \overline{f}_3 & D^\perp f_2 & = \gamma_{22} \otimes f_2 + \tau\zeta_1 \otimes f_3 \\
\nabla^\perp f_3 & = \textstyle \overline{\tau}\overline{\zeta}_1 \otimes f_2 + \frac{1}{2}\zeta_1 \otimes \overline{f}_2 + \gamma_{33} \otimes f_3 & D^\perp f_3 & = \overline{\tau}\overline{\zeta}_1 \otimes f_2 + \gamma_{33} \otimes f_3.
\end{align*}
Thus, the restriction of $\nabla^\perp$ to $\widetilde{N}^{1,0}\Sigma$ does not give a well-defined connection, whereas the restriction of $D^\perp$ to $\widetilde{N}^{1,0}$ does.  That is, we have a covariant derivative operator
$$D^\perp \colon \Gamma(\widetilde{N}^{1,0}) \to \Omega^{1}(\Sigma; \mathbb{C}) \otimes \Gamma(\widetilde{N}^{1,0}),$$
which gives $\widetilde{N}^{1,0}$ the holomorphic structure described in the previous paragraph.  Composing this $D^\perp$ with the projection $T\Sigma \otimes_{\mathbb{R}} \mathbb{C} \to T^{0,1}\Sigma$ gives the corresponding $\overline{\partial}$-operator
$$\overline{\partial}^\text{SU} \colon \Gamma(\widetilde{N}^{1,0}) \to \Omega^{0,1}(\Sigma) \otimes \Gamma(\widetilde{N}^{1,0}).$$

\subsubsection{A Second Complex Structure}  \label{sec:J2}

\indent \indent Let $u \colon \Sigma^2 \to \mathbb{S}^6$ be a holomorphic curve.  In terms of a local $T^2$-frame $(e_1, \ldots, e_6)$, we have:
\begin{align*}
\widetilde{J}e_3 & = e_4 & \widetilde{J}e_5 & = e_6.
\end{align*}
We now define a new complex structure $\widehat{J}$ on $N\Sigma$ by declaring
\begin{align*}
\widehat{J}e_3 & = e_4 & \widehat{J}e_5 & = -e_6.
\end{align*}
The following shows how $\nabla^\perp$ and $D^\perp$ relate to $\widehat{J}$, and should be contrasted with Proposition \ref{prop:JTilde}.

\begin{prop} Let $u \colon \Sigma \to \mathbb{S}^6$ be a holomorphic curve.  The following are equivalent: \\
\indent (i) $u$ is null-torsion. \\
\indent (ii) $\nabla^\perp \widehat{J} = 0$. \\
\indent (iii) $D^\perp \widehat{J} = 0$.
\end{prop}

\begin{proof} Directly from (\ref{eq:DFrames}), (\ref{eq:SU3Splitting}), and (\ref{eq:NablaFrames}), one can check that 
\begin{align*}
\nabla^\perp(\widehat{J}e_3) - \widehat{J}(\nabla^\perp e_3) = D^\perp(\widehat{J}e_3) - \widehat{J}(D^\perp e_3) & = -2\beta_{32} \otimes e_5 - 2 \alpha_{32} \otimes e_6 \\
\nabla^\perp(\widehat{J}e_4) - \widehat{J}(\nabla^\perp e_4) = D^\perp(\widehat{J}e_4) - \widehat{J}(D^\perp e_4) & = 2\alpha_{32} \otimes e_5 - 2 \beta_{32} \otimes e_6 \\
\nabla^\perp(\widehat{J}e_5) - \widehat{J}(\nabla^\perp e_5) = D^\perp(\widehat{J}e_5) - \widehat{J}(D^\perp e_5) & = 2\beta_{32} \otimes e_3 - 2\alpha_{32} \otimes e_4 \\
\nabla^\perp(\widehat{J}e_6) - \widehat{J}(\nabla^\perp e_6) = D^\perp(\widehat{J}e_6) - \widehat{J}(D^\perp e_6) & = 2\alpha_{32} \otimes e_3 + 2 \beta_{32} \otimes e_4
\end{align*}
Noting that $u$ is null-torsion if and only if $\alpha_{32} = \beta_{32} = 0$ proves the claim.
\end{proof}

\indent We now complexify $N\Sigma$ and extend $\widehat{J}$ by $\mathbb{C}$-linearity.  We have a decomposition of $N\Sigma \otimes_{\mathbb{R}} \mathbb{C}$ into $\widehat{J}$-eigenbundles, say
$$N\Sigma \otimes_{\mathbb{R}} \mathbb{C} = \widehat{N}^{1,0} \oplus \widehat{N}^{0,1}$$
where, for example, $\widehat{N}^{1,0} = \{\xi \in N\Sigma \otimes_{\mathbb{R}} \mathbb{C} \colon \widehat{J}\xi = i\xi \} = \textstyle \{ \frac{1}{2}(\eta - i\widehat{J}\eta) \colon \eta \in N\Sigma \}$.
As complex vector bundles, we have isomorphisms $(N\Sigma, \widehat{J}) \simeq (\widehat{N}^{1,0}, i) \simeq L_N \oplus L_B^*$.
In terms of a local $T^2$-frame, these isomorphisms are simply:
\begin{align*}
e_3 & \mapsto f_2 \mapsto (f_2) \oplus 0 & e_5 & \mapsto \overline{f}_3 \mapsto 0 \oplus (\!(\overline{f}_3)\!) \\
e_4 & \mapsto if_2 \mapsto (if_2) \oplus 0 & e_6 & \mapsto -i\overline{f}_3 \mapsto 0 \oplus (\!(-i\overline{f}_3)\!).
\end{align*}

\indent We now extend both $\nabla^\perp$ and $D^\perp$ by $\mathbb{C}$-linearity.  In terms of a local $T^2$-frame, we compute
\begin{align}
\nabla^\perp f_2 & = \textstyle \gamma_{22} \otimes f_2 + \tau\zeta_1 \otimes f_3 - \frac{1}{2}\zeta_1 \otimes \overline{f}_3 & D^\perp f_2 & = \gamma_{22} \otimes f_2 + \tau\zeta_1 \otimes f_3 \label{eq:NHatHolo1} \\
\nabla^\perp \overline{f}_3 & = \textstyle \frac{1}{2}\overline{\zeta}_1 \otimes f_2 +  \tau\zeta_1 \otimes \overline{f}_2 - \gamma_{33} \otimes \overline{f}_3 & D^\perp \overline{f}_3 & = \tau\zeta_1 \otimes \overline{f}_2  - \gamma_{33} \otimes \overline{f}_3. \label{eq:NHatHolo2}
\end{align}
Thus, if $u$ null-torsion (i.e., $\tau = 0$), the $\mathbb{C}$-linear extensions of $\nabla^\perp$ and $D^\perp$ both give well-defined connections on $\widehat{N}^{1,0}$, meaning that we have covariant derivative operators
\begin{align*}
\nabla^\perp, D^\perp \colon \Gamma(\widehat{N}^{1,0}) \to \Omega^{1}(\Sigma; \mathbb{C}) \otimes \Gamma(\widehat{N}^{1,0}).
\end{align*}
Composing these with the projection $T\Sigma \otimes_{\mathbb{R}} \mathbb{C} \to T^{0,1}\Sigma$ gives the corresponding $\overline{\partial}$-operators:
\begin{align*}
\overline{\partial}^\nabla, \overline{\partial}^D \colon \Gamma(\widehat{N}^{1,0}) \to \Omega^{0,1}(\Sigma) \otimes \Gamma(\widehat{N}^{1,0}).
\end{align*}
Explicitly, continuing to assume that $u$ is null-torsion:
\begin{align}
\overline{\partial}^\nabla f_2 & = \textstyle (\gamma_{22})^{0,1} \otimes f_2 & \overline{\partial}^D f_2 & = (\gamma_{22})^{0,1} \otimes f_2  \label{eq:NHatHolo3} \\
\overline{\partial}^\nabla \overline{f}_3 & = \textstyle \frac{1}{2}\overline{\zeta}_1 \otimes f_2 -  (\gamma_{33})^{0,1} \otimes \overline{f}_3 & \overline{\partial}^D \overline{f}_3 & = - (\gamma_{33})^{0,1} \otimes \overline{f}_3. \label{eq:NHatHolo4}
\end{align}
\indent The upshot is that if $u$ is null-torsion, we have two different holomorphic structures to consider on the complex bundle $\widehat{N}^{1,0}$.  The corresponding spaces of holomorphic sections will be denoted
\begin{align*}
H^0(\widehat{N}^{1,0}; \nabla^\perp) & = \{\xi \in \Gamma(\widehat{N}^{1,0}) \colon \overline{\partial}^\nabla\xi = 0 \} & H^0(\widehat{N}^{1,0}; D^\perp) & = \{\xi \in \Gamma(\widehat{N}^{1,0}) \colon \overline{\partial}^D\xi = 0 \}.
\end{align*}
As we will see, it is the bundle $(\widehat{N}^{1,0}, \nabla^\perp)$ that arises in the study of the second variation of area.  On the other hand, the bundle $(\widehat{N}^{1,0}, D^\perp)$ has the desirable feature that it splits holomorphically:

\begin{prop} \label{thm:HatHoloSplitting} If $u \colon \Sigma^2 \to \mathbb{S}^6$ is a null-torsion holomorphic curve, then: \\
\indent (a) The inclusion $L_N \hookrightarrow (\widehat{N}^{1,0}, \nabla^\perp)$ is holomorphic. \\
\indent (b) There is a holomorphic splitting
$$(\widehat{N}^{1,0}, D^\perp) \cong L_N \oplus L_B^*.$$
\end{prop}

\begin{proof} We have already seen that $\widehat{N}^{1,0} \simeq L_N \oplus L_B^*$ holds as complex vector bundles.  Equations (\ref{eq:NHatHolo3}) and (\ref{eq:NHatHolo4}) show that the inclusions
\begin{align*}
L_N & \hookrightarrow (\widehat{N}^{1,0}, \nabla^\perp) & L_N & \hookrightarrow (\widehat{N}^{1,0}, D^\perp) \\
& & L_B^* & \hookrightarrow (\widehat{N}^{1,0}, D^\perp)
\end{align*}
are all holomorphic.
\end{proof}

\indent To conclude this section, we use the holomorphic structure $\nabla^\perp$ on $\widehat{N}^{1,0}$ to define a notion of ``real-holomorphicity" for sections of $(N\Sigma, \widehat{J})$.  We say that $\eta \in \Gamma(N\Sigma)$ is $(\widehat{J}, \nabla^\perp)$-\textit{real holomorphic} if $\widehat{\eta}^{1,0} := \frac{1}{2}(\eta - i\widehat{J}\eta)$ is $\nabla^\perp$-holomorphic.  Defining the operator
\begin{align}
\widehat{\mathscr{D}} \colon \Gamma(N\Sigma) & \to \Omega^1(\Sigma) \otimes \Gamma(N\Sigma) \label{eq:RealHoloHat} \\
\widehat{\mathscr{D}}_X\eta & = \nabla^\perp_{JX}\eta - \nabla^\perp_X(\widehat{J}\eta) \nonumber
\end{align}
where $J := \widetilde{J}|_{T\Sigma}$ is the complex structure on $T\Sigma$, it is easy to check that for $Z = \frac{1}{2}(X + iJX) \in T^{0,1}\Sigma$, we have
\begin{align*}
\nabla^\perp_{Z}\widehat{\eta}^{1,0} = \textstyle \frac{i}{4}\left(\widehat{\mathscr{D}}_X\eta  + i\widehat{\mathscr{D}}_{JX}\eta \right)\!.
\end{align*}
Consequently, $\eta$ is $(\widehat{J}, \nabla^\perp)$-real holomorphic if and only if $\widehat{\mathscr{D}}_X\eta = 0$ for all $X \in T\Sigma$.  In other words,
\begin{align} \label{eq:RealHolo}
\{\eta \in \Gamma(N\Sigma) \colon \widehat{\mathscr{D}}\eta = 0\} & \cong H^0(\widehat{N}^{1,0}; \nabla^\perp)
\end{align}
as \textit{complex} vector spaces.

\begin{rmk} In the same way, one can speak of $(\widehat{J}, D^\perp)$-real holomorphicity, and define a corresponding operator $\widehat{\mathscr{D}}^D \colon \Gamma(N\Sigma) \to \Omega^1(\Sigma) \otimes \Gamma(N\Sigma)$.  However, we will not need this concept.
\end{rmk}

\section{The Second Variation of Area}

\indent \indent We now begin our study of the Jacobi operator of null-torsion holomorphic curves.  In $\S$4.1, we derive a second variation formula suited to the study of null-torsion holomorphic curves in $\mathbb{S}^6$, and in $\S$4.2 we use this formula to give a holomorphic interpretation of the $(-2)$-eigenspace of the Jacobi operator.  In $\S$4.3, we prove Theorem \ref{thm:M1}.

\subsection{A Second Variation Formula}

\indent \indent Let $u \colon \Sigma^2 \to \mathbb{S}^{2n}$ be a minimal surface in a round $2n$-sphere of constant curvature $1$, where $\Sigma$ is a compact oriented surface without boundary.  Let $u_t \colon \Sigma \to \mathbb{S}^{2n}$ be a variation of $F_0 = u$ with $\eta := \left.\frac{d}{dt}\right|_{t = 0} u_t$ a normal variation vector field.  As is well-known \cite[Chap I: $\S$9]{lawson80lectures}, the second variation of area is given by
$$\mathcal{Q}(\eta) := \left.\frac{d^2}{dt^2}\right|_{t = 0} \text{Area}(u_t) = \int_\Sigma \langle \mathcal{L}\eta, \eta \rangle$$
where the Jacobi operator $\mathcal{L} \colon \Gamma(N\Sigma) \to \Gamma(N\Sigma)$ is
$$\mathcal{L} = -\Delta^\perp - \mathcal{B} - \mathcal{R}$$
where, in a local orthonormal frame $(e_1, e_2)$ for $T\Sigma$, we have
\begin{align*}
\Delta^\perp\eta & =  \nabla^\perp_{e_i}\nabla^\perp_{e_i}\eta - \nabla^\perp_{\nabla^\top_{e_i}e_i}\eta & \mathcal{B}(\eta) & =  \left\langle \text{I\!I}(e_i, e_j), \eta \right\rangle \text{I\!I}(e_i, e_j) & \mathcal{R}(\eta) & =  (\overline{R}(\eta, e_i)e_i)^\perp.
\end{align*}
Note that $\Delta^\perp$ is simply the connection Laplacian for $\nabla^\perp$, the normal part of the Levi-Civita connection.  The spectrum of $-\Delta^\perp$ consists of non-negative real numbers. \\
\indent We now fix $\eta \in \Gamma(N\Sigma)$ and study the terms $\mathcal{R}(\eta)$, $\Delta^\perp\eta$, and $\mathcal{B}(\eta)$, in that order.  Let $\theta_\eta \in \Omega^1(\Sigma)$ denote the $1$-form
$$\theta_\eta(X) = \langle \nabla^\perp_X\eta, \eta \rangle.$$
Standard arguments show (see e.g., \cite{simons68}, \cite{ejiri83}, \cite{montiel-urbano}) that:

\begin{prop} \label{thm:RFormula} Let $u \colon \Sigma^2 \to \mathbb{S}^{2n}$ be an oriented minimal surface.  Then
\begin{equation} \label{eq:RFormula}
\mathcal{R}(\eta) = 2\eta.
\end{equation}
and
\begin{equation} \label{eq:Fact1}
-\langle \Delta^\perp\eta, \eta\rangle = \Vert \nabla^\perp \eta\Vert^2 + \mathrm{div}(\theta_\eta^\sharp).
\end{equation}
\end{prop}

\indent We now seek a formula for the term $\Vert \nabla^\perp \eta\Vert^2$.  For this, let $J$ denote the complex structure on $T\Sigma$ given by the metric and orientation, and equip $N\Sigma$ with a complex structure $I$ that is compatible with the metric.   Associated to $J$ and $I$, we let $\Theta_\eta \in \Omega^1(\Sigma)$ denote the $1$-form
\begin{align*}
\Theta_\eta(X) & = \langle \nabla^\perp_{JX}\eta, I\eta\rangle\!,
\end{align*}
and let $\mathscr{D} \colon \Gamma(N\Sigma) \to \Gamma(N\Sigma) \otimes \Omega^1(\Sigma)$ denote the operator
$$\mathscr{D}_X\eta = \nabla^\perp_{JX}\eta - \nabla^\perp_X(I\eta).$$
We now have:

\begin{prop} \label{thm:DPerpFormula} Let $u \colon \Sigma^2 \to \mathbb{S}^{2n}$ be an oriented minimal surface.  Let $I$ be any complex structure on $N\Sigma$ that is compatible with the metric.  If $\nabla^\perp I = 0$, then
\begin{equation} \label{eq:Fact2}
\Vert \nabla^\perp \eta \Vert^2 = \frac{1}{2} \Vert \mathscr{D}\eta \Vert^2 - \langle R^\perp_{12}\eta, I\eta \rangle - \mathrm{div}(\Theta_\eta^\sharp). \\
\end{equation}
\end{prop}
\begin{proof} We observe that
$$\Vert\mathscr{D}_X\eta\Vert^2 = \Vert \nabla^\perp_{JX}\eta \Vert^2 + \Vert \nabla^\perp_X(I\eta)\Vert^2 - 2 \left\langle \nabla^\perp_{JX}\eta, \nabla^\perp_{X}(I\eta) \right\rangle$$
Therefore, we have:
\begin{align*}
\Vert\mathscr{D}\eta \Vert^2 =  \left\Vert \mathscr{D}_{e_i}\eta\right\Vert^2 & =  \Vert \nabla^\perp_{Je_i}\eta \Vert^2 + \Vert \nabla^\perp_{e_i}(I\eta)\Vert^2 - 2\left\langle \nabla^\perp_{Je_i}\eta, \nabla^\perp_{e_i}(I\eta) \right \rangle \\
& = \Vert \nabla^\perp \eta \Vert^2 + \Vert \nabla^\perp (I\eta) \Vert^2 - 2  \left\langle \nabla^\perp_{Je_i}\eta, \nabla^\perp_{e_i}(I\eta) \right\rangle \\
& = 2\Vert \nabla^\perp \eta \Vert^2 - 2  \left\langle \nabla^\perp_{Je_i}\eta, \nabla^\perp_{e_i}(I\eta) \right\rangle
\end{align*}
using $\Vert \nabla^\perp(I\eta)\Vert = \Vert \nabla^\perp \eta\Vert$ in the last line.  Therefore,
$$\Vert \nabla^\perp \eta \Vert^2 = \frac{1}{2} \Vert \mathscr{D}\eta \Vert^2 +  \left\langle \nabla^\perp_{Je_i}\eta, \nabla^\perp_{e_i}(I\eta) \right\rangle\!.$$
To evaluate the last term, we compute
\begin{align*}
 \left\langle \nabla^\perp_{Je_i}\eta, \nabla^\perp_{e_i}(I\eta) \right\rangle & = - \left\langle \nabla^\perp_{e_i} \nabla^\perp_{Je_i}\eta, I\eta \right\rangle + e_i(\Theta_\eta(e_i)) \\
& = -\left\langle \nabla^\perp_{e_i} \nabla^\perp_{Je_i}\eta, I\eta \right\rangle + \Theta_\eta(\nabla^\top_{e_i}e_i) - \delta \Theta_\eta \\
& = -\langle R^\perp_{12}\eta, I\eta \rangle - \left\langle \nabla^\perp_{[e_1, e_2]}\eta, I\eta \right\rangle + \Theta_\eta(\nabla^\top_{e_i}e_i) - \delta \Theta_\eta \\
& = -\langle R^\perp_{12}\eta, I\eta \rangle  - \delta \Theta_\eta
\end{align*}
where $\delta$ is the codifferential, and in the last line we used that $\nabla^\top$ is torsion-free and commutes with $J$.  Thus, we have shown that
$$\Vert \nabla^\perp \eta \Vert^2 = \frac{1}{2} \Vert \mathscr{D}\eta \Vert^2 - \langle R^\perp_{12}\eta, I\eta \rangle - \delta \Theta_\eta$$
This gives the result.
\end{proof}

\indent Finally, we need a formula for $\mathcal{B}(\eta)$.  For this, we specialize to the case of holomorphic curves in $\mathbb{S}^6$ and recall the complex structure $\widehat{J}$ on $N\Sigma$ defined in $\S$\ref{sec:J2}.

\begin{prop} \label{thm:BFormula} Let $u \colon \Sigma^2 \to \mathbb{S}^6$ be a holomorphic curve.  Then
\begin{equation} \label{eq:BFormula1}
\mathcal{B}(\eta) = \widehat{J} R^\perp_{12}\eta
\end{equation}
so that
\begin{equation}
\label{eq:BFormula2}
\langle\mathcal{B}(\eta),\eta \rangle = -\langle R^\perp_{12}\eta, \widehat{J}\eta \rangle.
\end{equation}
\end{prop}
\begin{proof} Let $(e_1, \ldots, e_6)$ be a $T^2$-adapted frame.  By (\ref{eq:SecondFundSym}) and the fact that $T^2$-adapted frames have $\mu = 0$, we have
\begin{align*}
\text{I\!I}(e_1, e_1) & = \kappa_1 e_3 + \kappa_2 e_4 \\
\text{I\!I}(e_1, e_2) & = -\kappa_2 e_3 + \kappa_1 e_4.
\end{align*}
It follows that
\begin{align*}
\mathcal{B}(e_3) & = 2|\kappa|^2 e_3 & \mathcal{B}(e_5) & = 0 \\
\mathcal{B}(e_4) & = 2|\kappa|^2 e_4 & \mathcal{B}(e_6) & = 0.
\end{align*}
On the other hand, the Ricci equation (\ref{eq:RicciEq}) implies that
\begin{align}
R^\perp_{12}(e_3) & = -2|\kappa|^2 \widehat{J}e_3 & R^\perp_{12}(e_5) & = 0  \label{eq:RPerp1} \\
R^\perp_{12}(e_4) & = -2|\kappa|^2 \widehat{J}e_4 & R^\perp_{12}(e_6) & = 0. \label{eq:RPerp2}
\end{align}
Therefore, $\mathcal{B}(\eta) = \widehat{J} R^\perp_{12}\eta$, so that $\langle \mathcal{B}(\eta), \eta \rangle = \langle \widehat{J}R^\perp_{12}\eta, \eta \rangle = -\langle R^\perp_{12}\eta, \widehat{J}\eta \rangle$.
\end{proof}

\indent We now intend to combine Propositions \ref{thm:RFormula}, \ref{thm:DPerpFormula}, and \ref{thm:BFormula} to arrive at a second variation formula for holomorphic curves in $\mathbb{S}^6$.  To do this, notice that Proposition \ref{thm:DPerpFormula} required the choice of a complex structure $I$ on $N\Sigma$ satisfying $\nabla^\perp I = 0$.  In $\S$\ref{sec:StructuresOnNormal}, we observed that $\widehat{J}$ satisfies $\nabla^\perp\widehat{J} = 0$ if and only if $u$ is null-torsion (whereas $\widetilde{J}$ never has this property).  Therefore, restricting to the null-torsion case and taking $I := \widehat{J}$ in Proposition \ref{thm:DPerpFormula}, we deduce:

\begin{cor} \label{thm:SecondVar} Let $u \colon \Sigma^2 \to \mathbb{S}^6$ be a null-torsion holomorphic curve, where $\Sigma$ is a closed surface.  For $\eta \in \Gamma(N\Sigma)$, we have:
\begin{equation}
\mathcal{Q}(\eta) = \int_\Sigma \frac{1}{2}\Vert \widehat{\mathscr{D}}\eta \Vert^2 - 2 \Vert \eta \Vert^2 \label{eq:SecondVar}
\end{equation}
recalling from (\ref{eq:RealHoloHat}) that $\widehat{\mathscr{D}}_X\eta := \nabla^\perp_{JX}\eta - \nabla^\perp_X(\widehat{J}\eta)$.
\end{cor}

\begin{proof} Using (\ref{eq:RFormula}) and (\ref{eq:BFormula2}), we have
\begin{align*}
\langle \mathcal{L}\eta, \eta \rangle & = -\langle \Delta^\perp \eta, \eta \rangle - \langle \mathcal{B}(\eta), \eta \rangle - \langle \mathcal{R}(\eta), \eta \rangle \\
& = -\langle \Delta^\perp \eta, \eta \rangle + \langle R^\perp_{12} \eta, \widehat{J}\eta \rangle - 2 \Vert \eta \Vert^2.
\end{align*}
Next, using (\ref{eq:Fact1}) and (\ref{eq:Fact2}) with the choice $I = \widehat{J}$, we have
\begin{align*}
\langle \mathcal{L}\eta, \eta \rangle & = \frac{1}{2}\Vert \widehat{\mathscr{D}}\eta \Vert^2 - \langle R^\perp_{12}\eta, \widehat{J}\eta \rangle - \text{div}(\Theta^\sharp_\eta) + \text{div}(\theta^\sharp_\eta) + \langle R^\perp_{12} \eta, \widehat{J}\eta \rangle - 2 \Vert \eta \Vert^2 \\
& = \frac{1}{2}\Vert \widehat{\mathscr{D}}\eta \Vert^2 - \text{div}(\Theta^\sharp_\eta) + \text{div}(\theta^\sharp_\eta) - 2 \Vert \eta \Vert^2.
\end{align*}
Integrating both sides and using Stokes' Theorem completes the proof.
\end{proof}

\indent Analogues of the second variation formula (\ref{eq:SecondVar}) have been observed in several other contexts.  For example, a version of (\ref{eq:SecondVar}) was obtained by Simons \cite[p. 78]{simons68} for complex submanifolds of K\"{a}hler manifolds, by Ejiri \cite[Lemma 3.2]{ejiri83} for minimal $2$-spheres in $\mathbb{S}^{2n}$, by Micallef-Wolfson \cite[p. 264]{micallef-wolfson} for minimal Lagrangians in negative K\"{a}hler-Einstein $4$-manifolds, and by Montiel-Urbano \cite[p. 2259]{montiel-urbano} for superminimal surfaces in self-dual Einstein $4$-manifolds.

\subsection{The First Eigenvalue}

\indent \indent Let $u \colon \Sigma^2 \to \mathbb{S}^6$ be a null-torsion holomorphic curve.  Let $\eta \in \Gamma(N\Sigma)$ be an eigenvector for $\mathcal{L}$, so that $\mathcal{L}\eta = \lambda \eta$ for some $\lambda \in \mathbb{R}$.  Then $\langle \mathcal{L}\eta, \eta \rangle = \lambda \Vert \eta \Vert^2$, so Corollary \ref{thm:SecondVar} gives
$$\int_\Sigma \frac{1}{2}\Vert \widehat{\mathscr{D}}\eta \Vert^2 - 2 \Vert \eta \Vert^2 = \mathcal{Q}(\eta) = \int_\Sigma \lambda \Vert \eta \Vert^2.$$
Rearranging, we obtain
$$\int_\Sigma \frac{1}{2}\Vert \widehat{\mathscr{D}}\eta \Vert^2 - (\lambda + 2) \Vert \eta \Vert^2 = 0.$$
Since $\Vert \widehat{\mathscr{D}}\eta \Vert^2 \geq 0$, we deduce that $\lambda \geq -2$.  Moreover, we see that $\lambda = -2$ if and only if $\eta \in \Gamma(N\Sigma)$ satisfies $\widehat{\mathscr{D}}\eta = 0$, i.e., if and only if $\eta$ is $(\widehat{J}, \nabla^\perp)$-real holomorphic.  Recalling (\ref{eq:RealHolo}), this proves:

\begin{prop} \label{thm:LowestEigen} The lowest eigenvalue of the Jacobi operator satisfies $\lambda_1 \geq -2$.  The multiplicity $m_1$ of the eigenvalue $\lambda = -2$ is
$$m_1 = \dim_{\mathbb{R}}\{ \eta \in \Gamma(N\Sigma) \colon \widehat{\mathscr{D}}\eta = 0\} = 2\,h^0(\widehat{N}^{1,0}; \nabla^\perp).$$
\end{prop}

\indent We now use the Riemann-Roch Theorem to derive a lower bound for $m_1$.  This lower bound will show, in particular, that $m_1 \geq 4$, so that the lowest eigenvalue of $\mathcal{L}$ is always $\lambda_1 = -2$.

\begin{prop} \label{thm:LowerBound} The multiplicity $m_1$ satisfies
$$m_1 = 4d + 2h^1(\widehat{N}^{1,0}; \nabla^\perp) \geq 4d.$$
\end{prop}

\begin{proof} By the Riemann-Roch Theorem for rank 2 vector bundles, the isomorphism $\widehat{N}^{1,0} \simeq L_N \oplus L_B^*$ of complex vector bundles, and Proposition \ref{thm:c1LineBundles}, we have:
\begin{align*}
h^0(\widehat{N}^{1,0}; \nabla^\perp) - h^1(\widehat{N}^{1,0}; \nabla^\perp) & = \deg(\widehat{N}^{1,0}) + \chi(\Sigma) \\
& = c_1(L_N) - c_1(L_B) + \chi(\Sigma) \\
& = -\chi(\Sigma) + d + d + \chi(\Sigma) \\
& = 2d.
\end{align*}
The result follows.
\end{proof}

\subsection{The Multiplicity of the First Eigenvalue: Proof of Theorem 1.1}

\indent \indent In Proposition \ref{thm:LowerBound}, we obtained a lower bound for $m_1$.  Our aim is to prove the following upper bound for $m_1$.

\begin{prop} \label{thm:UpperBound} Let $u \colon \Sigma^2 \to \mathbb{S}^6$ be a null-torsion holomorphic curve.  Then 
$$h^0(\widehat{N}^{1,0}; \nabla^\perp) \leq h^0(L_N) + h^0(L_B^*).$$
\end{prop}

Accepting this proposition on faith for a moment, we now show how it implies Theorem 1.1. \\

\begin{proof} Let $u \colon \Sigma^2 \to \mathbb{S}^6$ be a null-torsion holomorphic curve satisfying $g < \frac{1}{2}(d+2)$, or equivalently, $d > 2g - 2$.  Using Proposition \ref{thm:LowestEigen}, followed by Proposition \ref{thm:UpperBound}, and finally Proposition \ref{thm:HoloSections}(a) and (b), we have the upper bound
\begin{align*}
m_1 = 2h^0(\widehat{N}^{1,0}; \nabla^\perp) & \leq 2 \left(h^0(L_N) + h^0(L_B^*) \right) = 2 \left( 2d + h^0(L_B \otimes K_\Sigma) \right) = 4d.
\end{align*}
Coupled with the lower bound of Proposition \ref{thm:LowerBound}, we deduce that $m_1 = 4d = \frac{A}{\pi}$.
\end{proof}

\subsubsection{Proof of Proposition \ref{thm:UpperBound}}

\indent \indent Let $u \colon \Sigma^2 \to \mathbb{S}^6$ be a null-torsion holomorphic curve.  On the complex vector bundle $\widehat{N}^{1,0}$, we recall the two $\overline{\partial}$-operators
\begin{align*}
\overline{\partial}^\nabla, \overline{\partial}^D \colon \Gamma(\widehat{N}^{1,0}) \to \Omega^{0,1}(\Sigma) \otimes \Gamma(\widehat{N}^{1,0}).
\end{align*}
Let $S \colon \Gamma(\widehat{N}^{1,0}) \to \Omega^{0,1}(\Sigma) \otimes \Gamma(\widehat{N}^{1,0})$ denote the difference tensor, i.e.,
$$S(\xi) := \overline{\partial}^\nabla\xi - \overline{\partial}^D\xi$$
for smooth sections $\xi \in \Gamma(\widehat{N}^{1,0})$.  Since $S$ is a tensor, it can be viewed as a pointwise operator, i.e., as a bundle map $\widehat{N}^{1,0} \to \Lambda^{0,1}\Sigma \otimes \widehat{N}^{1,0}$.  To understand $S$ in more detail, let $(f_2, \overline{f}_3)$ be a $T^2$-frame for $\widehat{N}^{1,0}$ at a point $p \in \Sigma$.  By (\ref{eq:NHatHolo3}) and (\ref{eq:NHatHolo4}), we have
\begin{align*}
S(f_2) & = 0 & S(\overline{f}_3) & = \textstyle \frac{1}{2}\overline{\zeta}_1 \otimes f_2.
\end{align*}
Thus, using the Hermitian vector bundle isomorphism $\widehat{N}^{1,0} \simeq L_N \oplus L_B^*$, we can regard $S$ as a map
$$S \colon L_N \oplus L_B^* \to \Lambda^{0,1}\Sigma \otimes L_N$$
such that $S|_{L_N} = 0$. \\

\indent Let $\xi \in \Gamma(\widehat{N}^{1,0})$ be a smooth section.  Write $\xi = \xi_2 + \xi_3$, where $\xi_2 \in \Gamma(L_N)$ and $\xi_2 \in \Gamma(L_B^*)$.  The condition that $\xi$ be $\nabla^\perp$-holomorphic is equivalent to
$$\overline{\partial}^D\xi = -S(\xi),$$
which (by decomposing into $L_N$ and $L_B^*$ components) is in turn is equivalent to the system
\begin{align*}
\overline{\partial}^D \xi_2 & = -S(\xi_3) \\
\overline{\partial}^D \xi_3 & = 0.
\end{align*}
Now, by Proposition \ref{thm:HatHoloSplitting}(b), there is a holomorphic isomorphism $(\widehat{N}^{1,0}, D^\perp) \cong L_N \oplus L_B^*$.  Thus, the $\nabla^\perp$-holomorphicity condition can finally be rewritten as
\begin{align}
\overline{\partial}^{L_N} \xi_2 & = -S(\xi_3) \label{eq:DHolo1} \\
\overline{\partial}^{L_B^*} \xi_3 & = 0, \label{eq:DHolo2}
\end{align}
where $\overline{\partial}^{L_N}$ and $\overline{\partial}^{L_B^*}$ are the respective $\overline{\partial}$-operators on $L_N$ and $L_B^*$.

\indent The upshot is that the system (\ref{eq:DHolo1})-(\ref{eq:DHolo2}) is decoupled, making it easy to count its solutions.  Indeed, the solution space of (\ref{eq:DHolo2}) is $H^0(L_B^*)$, a $\mathbb{C}$-vector space of complex dimension $h^0(L_B^*)$.  Moreover, for each $\xi_3 \in H^0(L_B^*)$, the set of solutions to (\ref{eq:DHolo1}) is either empty or an affine space of complex dimension $h^0(L_N)$.  Geometrically, the set of solutions to (\ref{eq:DHolo1})-(\ref{eq:DHolo2}) can be viewed as a bundle of complex $h^0(L_N)$-dimensional affine spaces over the set of $\xi_3 \in H^0(L_B^*)$ for which (\ref{eq:DHolo1}) has a solution.  In conclusion, the set of $\nabla^\perp$-holomorphic sections has complex dimension at most $h^0(L_N) + h^0(L_B^*)$.  This proves Proposition \ref{thm:UpperBound}.

\section{The Nullity: Proof of Theorem 1.2}

\indent \indent Let $u \colon \Sigma^2 \to \mathbb{S}^6$ be a null-torsion holomorphic curve of area $A = 4\pi d$.  The aim of this section is to prove the following lower bound on the nullity of the Jacobi operator $\mathcal{L}$ of $u$:
$$\text{Nullity}(u) \geq 2d + \chi(\Sigma).$$
The idea of the proof is to identify a subspace of $\text{Null}(u) := \{\eta \in \Gamma(N\Sigma) \colon \mathcal{L}\eta = 0\}$ that is isomorphic to $H^0(K_\Sigma^* \otimes L_N)$.  The dimension of $H^0(K_\Sigma^* \otimes L_N)$ will then be estimated via Riemann-Roch. \\
\indent Our method in this section is not original.  Indeed, our calculations are direct analogues those in Montiel and Urbano's study \cite{montiel-urbano} of superminimal surfaces in self-dual Einstein $4$-manifolds.  To ease notation, we enact the following conventions: \\

\noindent \textbf{Convention:} For the remainder of this work, we let $J$ denote the complex structure on $u^*(T\mathbb{S}^6) = T\Sigma \oplus N\Sigma$ that on $T\Sigma$ is given by the metric and orientation, and on $N\Sigma$ is given by $\widehat{J}$.  Recall that as complex bundles, we have an isomorphism
$$\widehat{N}^{1,0} \simeq (N\Sigma, J) = E_N \oplus E_B^*.$$
As \textit{real} vector bundles, we have $N\Sigma \simeq E_N \oplus E_B$, and we will use the notation $\eta = \eta^N + \eta^B$ for the decomposition of a normal vector $\eta \in N\Sigma$ into its $E_N$ and $E_B$ components. \\

\noindent \textbf{Convention:} On the complex bundle $\widehat{N}^{1,0}$, the only holomorphic structure we will need from now on is $\overline{\partial}^\nabla$.  Thus, we will abbreviate ``$\nabla^\perp$-holomorphic" as ``holomorphic," and abbreviate ``$(\widehat{J}, \nabla^\perp)$-real holomorphic" as ``real-holomorphic."  Noting that holomorphic sections of $K_\Sigma^* \otimes \widehat{N}^{1,0}$ are in bijection with real-holomorphic sections of $T^*\Sigma \otimes N\Sigma$, we may identify
\begin{align*}
\Gamma(K_\Sigma^* \otimes L_N) & \cong \{\alpha \in \Gamma(T^*\Sigma \otimes E_N) \colon \alpha \circ J = -J \circ \alpha\} \\
H^0(K_\Sigma^* \otimes L_N) & \cong \{\alpha \in \Gamma(T^*\Sigma \otimes E_N) \colon \alpha \circ J = -J \circ \alpha \text{ and } \alpha \text{ real-holomorphic}\}.
\end{align*}
These identifications will frequently be made without comment.

\subsection{Preliminaries}

\indent \indent Recall that the Levi-Civita connection $\overline{\nabla}$ on $\mathbb{S}^6$ gives a tangential connection $\nabla^\top$on $T\Sigma$ and a normal connection $\nabla^\perp$ on $N\Sigma$.  These give an induced connection $\widetilde{\nabla}$ on $T^*\Sigma \otimes N\Sigma$, and an induced connection $\widehat{\nabla}$ on $T^*\Sigma \otimes T^*\Sigma \otimes N\Sigma$.  Explicitly, for $\alpha \in \Gamma(T^*\Sigma \otimes N\Sigma)$, we have
\begin{equation}
(\widetilde{\nabla}_Y\alpha)(X) := \nabla^\perp_Y(\alpha(X)) - \alpha(\nabla^\top_YX),
\label{eq:Alpha}
\end{equation}
and for $\beta \in \Gamma(T^*\Sigma \otimes T^*\Sigma \otimes N\Sigma)$, we have
\begin{equation}
(\widehat{\nabla}_Z\beta)(X,Y) := \nabla^\perp_Z(\beta(X,Y)) - \beta(\nabla^\top_ZX, Y) - \beta(X, \nabla^\top_ZY).
\label{eq:Beta}
\end{equation}
We remark that if $\beta$ is a symmetric $2$-tensor (i.e., $\beta(X,Y) = \beta(Y,X)$ for all $X,Y \in T\Sigma$), then $\widehat{\nabla}_Z\beta$ is also a symmetric $2$-tensor. \\

\indent For $\alpha \in \Gamma(T^*\Sigma \otimes N\Sigma)$, we recall that its \textit{second covariant derivative} $\widetilde{\nabla}^2_{X,Y}\alpha \in \Gamma(T^*\Sigma \otimes N\Sigma)$ at $X,Y \in T\Sigma$ is given by:
\begin{align*}
\widetilde{\nabla}^2_{XY}\alpha & := (\widehat{\nabla}_X\widetilde{\nabla}\alpha)(Y, \cdot) = \widetilde{\nabla}_X\widetilde{\nabla}_Y\alpha - \widetilde{\nabla}_{\nabla^\top_XY}\alpha.
\end{align*}
We also recall the \textit{Ricci identity}
\begin{equation}
\widetilde{\nabla}^2_{XY}\alpha = \widetilde{\nabla}^2_{YX}\alpha + \widetilde{R}_{XY}\alpha \label{eq:RicciIdentity}
\end{equation}
where $\widetilde{R}$ is the curvature of $\widetilde{\nabla}$.  A straightforward calculation shows that
\begin{equation}
(\widetilde{R}_{XY}\alpha)(Z) = R^\perp_{XY}(\alpha(Z)) - \alpha(R^\top_{XY}(Z)). \label{eq:RTilde}
\end{equation} \\
\indent We let $(\nabla^\perp)^* \colon \Gamma(T^*\Sigma \otimes N\Sigma) \to \Gamma(N\Sigma)$ denote the formal adjoint of $\nabla^\perp$, so that
$$\int_\Sigma \left\langle \nabla^\perp \xi, \alpha \right\rangle = \int_\Sigma \left\langle \xi, (\nabla^\perp)^*\alpha \right\rangle \ \ \text{ for all } \xi \in \Gamma(N\Sigma), \, \alpha \in \Gamma(T^*\Sigma \otimes N\Sigma).$$
In terms of a local orthonormal frame $(e_1, e_2)$ on $T\Sigma$, one can compute $(\nabla^\perp)^*\alpha$ via the well-known formula:
\begin{equation}
(\nabla^\perp)^*\alpha = -(\widetilde{\nabla}_{e_i}\alpha)(e_i).
\label{eq:CodiffAlpha}
\end{equation}
Similarly, we let $\widetilde{\nabla}^* \colon \Gamma(T^*\Sigma \otimes T^*\Sigma \otimes N\Sigma) \to \Gamma(T^*\Sigma \otimes N\Sigma)$ denote the formal adjoint of $\widetilde{\nabla}$.  Again, in terms of a local orthonormal frame $(e_1, e_2)$ on $T\Sigma$, one has the formula
\begin{equation}
\widetilde{\nabla}^*\beta = -(\widehat{\nabla}_{e_i}\beta)(e_i, \cdot).
\label{eq:CodiffBeta}
\end{equation}

\subsection{Strategy of Proof}

\indent \indent For a fixed $\eta \in \Gamma(N\Sigma)$, we consider the section $\Psi_\eta \in \Gamma(T^*\Sigma \otimes N\Sigma)$ given by
$$\Psi_\eta(X) := \widehat{\mathscr{D}}_X\eta = \nabla^\perp_{JX}\eta - J(\nabla^\perp_X \eta).$$
The basic properties of $\Psi_\eta$ are given by the following two lemmas.  Verifying Lemma \ref{thm:A} is straightforward; we will prove only Lemma \ref{thm:Codiff1}.

\begin{lem} \label{thm:A} We have: \\
\indent (a) $\eta$ is  real-holomorphic $\iff$ $\Psi_\eta = 0$. \\
\indent (b) $\Psi_\eta \circ J = -J \circ \Psi_\eta$. \\
\indent (c) We have
$$(\widetilde{\nabla}_Y\Psi_\eta)(JX) = -J\left[ (\widetilde{\nabla}_Y\Psi_\eta)(X) \right]\!.$$
\end{lem}

\begin{lem} \label{thm:Codiff1} We have:
$$(\nabla^\perp)^* \Psi_\eta = -J(\mathcal{L}\eta+2\eta).$$
\end{lem}

\begin{proof}
Using (\ref{eq:CodiffAlpha}) and (\ref{eq:Alpha}), we calculate 
\begin{align*}
(\nabla^\perp)^*\Psi_\eta = -(\widetilde{\nabla}_{e_i}\Psi_\eta)(e_i) & = -\nabla^\perp_{e_i}(\Psi_\eta(e_i)) + \Psi_\eta(\nabla^\top_{e_i}e_i) \\
& = -\nabla^\perp_{e_i}\nabla^\perp_{Je_i}\eta + J\nabla^\perp_{e_i}\nabla^\perp_{e_i}\eta + \nabla^\perp_{J(\nabla^\top_{e_i}e_i)}\eta - J(\nabla^\perp_{\nabla^\top_{e_i}e_i}\eta) \\
& = -\nabla^\perp_{e_i}\nabla^\perp_{Je_i}\eta + \nabla^\perp_{J(\nabla^\top_{e_i}e_i)}\eta + J \Delta^\perp \eta.
\end{align*}
Now, using that $\nabla^\perp_{e_i}\nabla^\perp_{Je_i} = R^\perp_{12} + \nabla^\perp_{[e_1, e_2]}$ and that $J(\nabla^\top_{e_i}e_i) = [e_1, e_2]$, we obtain
\begin{align*}
(\nabla^\perp)^*\Psi_\eta & = -R^\perp_{12}\eta + J \Delta^\perp \eta \\
& = -R^\perp_{12}\eta + J(-\mathcal{L}\eta - 2\eta) - J(\mathcal{B}\eta)
\end{align*}
where we used that $\Delta^\perp \eta = -\mathcal{L}\eta - \mathcal{B}\eta - 2\eta$.  Finally, using (\ref{eq:BFormula1}), we arrive at the result.
\end{proof}

\indent Let $\text{Null}(u) = \{\eta \in \Gamma(N\Sigma) \colon \mathcal{L}\eta = 0\}$ denote the null space of the Jacobi operator $\mathcal{L}$.  Consider the linear map
\begin{align*}
G \colon \{\eta \in \text{Null}(u) \colon (\Psi_\eta)^B = 0\} & \to \Gamma(K_\Sigma^* \otimes L_N) \\
\eta & \mapsto \widehat{\mathscr{D}}\eta = \Psi_\eta.
\end{align*}
Observe that $G$ is injective.  Indeed, if $G(\eta) = 0$, then $\Psi_\eta = 0$, so $\eta$ is real-holomorphic, so $\mathcal{L}\eta = -2\eta$, but $\mathcal{L}\eta = 0$, so $\eta = 0$.  Our main claim in this section concerns the image of $G$: 

\begin{prop} \label{thm:GSurj} The image of $G$ is equal to $H^0(K_\Sigma^* \otimes L_N)$. 
\end{prop}

\indent Accepting Proposition \ref{thm:GSurj} on faith for a moment, we see that $G$ gives an isomorphism
\begin{align*}
\{\eta \in \text{Null}(u) \colon (\Psi_\eta)^B = 0\} \cong H^0(K_\Sigma^* \otimes L_N).
\end{align*}
From this isomorphism, we now deduce Theorem 1.2:

\begin{proof} Recall from Proposition \ref{thm:c1LineBundles} that $c_1(L_N) = -\chi(\Sigma) + d$ and $c_1(K_\Sigma^*) = \chi(\Sigma)$, and hence $c_1(K_\Sigma^* \otimes L_N) = d$. By Proposition \ref{thm:GSurj}, we now estimate
\begin{align*}
\text{Nullity}(u) \geq \dim_{\mathbb{R}}\{\eta \in \text{Null}(u) \colon (\Psi_\eta)^B = 0\} & = \dim_{\mathbb{R}}[H^0(K_\Sigma^* \otimes L_N)] \\
& = 2 h^0(K_\Sigma^* \otimes L_N) \\
& = 2 h^1(K_\Sigma^* \otimes L_N) + 2c_1(K_\Sigma^* \otimes L_N) + \chi(\Sigma) \\
& = 2 h^1(K_\Sigma^* \otimes L_N) + 2d + \chi(\Sigma),
\end{align*}
where we used Riemann-Roch in the second-to-last step.  Finally, using $h^1(K_\Sigma^* \otimes L_N) \geq 0$, we conclude the result.
\end{proof}

\begin{rmk} The estimate $h^1(K_\Sigma^* \otimes L_N) \geq 0$ can be slightly sharpened.  Indeed, by Serre Duality, we have $h^1(K_\Sigma^* \otimes L_N) = h^0(K_\Sigma \otimes K_\Sigma \otimes L_N^*)$, and we compute $c_1(K_\Sigma \otimes K_\Sigma \otimes L_N^*) = -\chi(\Sigma) - d$.  Therefore, if $d > 2g - 2$ --- which, by Proposition \ref{thm:GenusDegree} holds if $g \leq 6$ --- then $K_\Sigma \otimes K_\Sigma \otimes L_N^*$ is a negative line bundle, and hence $h^1(K_\Sigma^* \otimes L_N) = 0$.
\end{rmk}

\indent The remainder of this section consists of a proof of Proposition \ref{thm:GSurj}, which naturally divides into two halves.  That is, in Proposition \ref{thm:GSurjA}, we will show that $\text{Im}(G) \subset H^0(K_\Sigma^* \otimes L_N)$, and in Proposition \ref{thm:GSurjB}, we will show that $H^0(K_\Sigma^* \otimes L_N) \subset \text{Im}(G)$.

\subsection{Technical Lemmas}

\indent \indent We begin by measuring the extent to which the smooth section $\Psi_\eta \in \Gamma(T^*\Sigma \otimes N\Sigma)$ might fail to be real-holomorphic.  So, for a fixed $\eta \in \Gamma(N\Sigma)$, we consider the section $\Omega_\eta \in \Gamma(T^*\Sigma \otimes T^*\Sigma \otimes N\Sigma)$ given by
$$\Omega_\eta(X,Y) := (\widetilde{\nabla}_{JX} \Psi_\eta)(JY) - (\widetilde{\nabla}_X \Psi_\eta)(Y).$$
We now establish the basic properties of $\Omega_\eta$ by analogy with Lemmas \ref{thm:A} and \ref{thm:Codiff1}.  The analogue of Lemma \ref{thm:A} is easy:

\begin{lem} \label{thm:B} We have: \\
\indent (a) $\Psi_\eta$ is real-holomorphic $\iff$ $\Omega_\eta = 0$. \\
\indent (b) $\Omega_\eta(JX, JY) = -\Omega_\eta(X,Y)$.  Therefore, $\Omega_\eta$ is an $N\Sigma$-valued symmetric $2$-tensor on $T\Sigma$ of trace zero. \\
\indent (c) We have the identity:
\begin{equation*}
(\widehat{\nabla}_{e_i} \Omega_\eta)(v, e_i) = (\widetilde{\nabla}^2_{e_i, Jv} \Psi_\eta)(Je_i) - (\widetilde{\nabla}^2_{e_i, v} \Psi_\eta)(e_i).
\end{equation*}
\end{lem}

\begin{proof} (a) This is straightforward and left to the reader. \\ 
 \indent (b) Directly from the definition of $\Omega_\eta$, we have
$$\Omega_\eta(JX,JY) = -\Omega_\eta(X,Y).$$
Thus, letting $(e_1, e_2)$ denote an oriented orthonormal frame on $T\Sigma$, we have both $\Omega_\eta(e_2, e_2) = -\Omega_\eta(e_1, e_1)$ and $\Omega_\eta(e_1, e_2) = \Omega_\eta(e_2, e_1)$, so $\Omega_\eta$ is an $N\Sigma$-valued symmetric $2$-tensor of trace zero. \\
\indent (c) Using (\ref{eq:Alpha}) and the fact that $\nabla^\top_X(JY) = J\nabla^\top_XY$ for all $X,Y$, the first term on the right is
\begin{align*}
(\widetilde{\nabla}^2_{e_i, Jv}\Psi_\eta)(Je_i) & = (\widetilde{\nabla}_{e_i}\widetilde{\nabla}_{Jv}\Psi_\eta)(Je_i) - (\widetilde{\nabla}_{\nabla^\top_{e_i}Jv}\Psi_\eta)(Je_i) \\
& = \nabla_{e_i}^\perp((\widetilde{\nabla}_{Jv}\Psi_\eta)(Je_i)) - (\widetilde{\nabla}_{Jv}\Psi_\eta)(J\nabla^\top_{e_i}e_i) - (\widetilde{\nabla}_{J\nabla^\top_{e_i}v}\Psi_\eta)(Je_i),
\end{align*}
and similarly, the second term on the right is
\begin{align*}
(\widetilde{\nabla}^2_{e_i, v}\Psi_\eta)(e_i) & = (\widetilde{\nabla}_{e_i}\widetilde{\nabla}_{v}\Psi_\eta)(e_i) - (\widetilde{\nabla}_{\nabla^\top_{e_i}v}\Psi_\eta)(e_i) \\
& = \nabla_{e_i}^\perp((\widetilde{\nabla}_{v}\Psi_\eta)(e_i)) - (\widetilde{\nabla}_{v}\Psi_\eta)(\nabla^\top_{e_i}e_i) - (\widetilde{\nabla}_{\nabla^\top_{e_i}v}\Psi_\eta)(e_i).
\end{align*}
On the other hand, using (\ref{eq:Beta}) and the definition of $\Omega_\eta$, the left side of the desired identity is:
\begin{align*}
(\widetilde{\nabla}_{e_i} \Omega_\eta)(v, e_i) & = \nabla^\perp_{e_i}[\Omega_\eta(v,e_i)] - \Omega_\eta(\nabla^\top_{e_i}v, e_i) - \Omega_\eta(v, \nabla^\top_{e_i}e_i) \\
& = \nabla_{e_i}^\perp( (\widetilde{\nabla}_{Jv} \Psi_\eta)(Je_i)) - \nabla_{e_i}^\perp((\widetilde{\nabla}_v \Psi_\eta)(e_i) ) -  (\widetilde{\nabla}_{J{\nabla^\top_{e_i}v}} \Psi_\eta)(Je_i)  \\
& \ \ \ \ \ \ \ + (\widetilde{\nabla}_{\nabla^\top_{e_i}v} \Psi_\eta)(e_i) -  (\widetilde{\nabla}_{Jv} \Psi_\eta)(J\nabla^\top_{e_i}e_i) + (\widetilde{\nabla}_v \Psi_\eta)(\nabla^\top_{e_i}e_i).
\end{align*}
Comparing terms proves the lemma. \end{proof}

\indent The identities in the following lemma are straightforward to prove, but rather tedious.  To streamline discussion, their verifications are deferred to the Appendix.

\begin{lem} \label{thm:D}
Let $\alpha \in \Gamma(T^*\Sigma \otimes N\Sigma)$ satisfy $\alpha \circ J = -J \circ \alpha$.  Let $(e_1, e_2)$ be a local oriented orthonormal frame on $\Sigma$.  Then for all $v \in T\Sigma$:
\begin{equation} \label{eq:Da}
(\widetilde{\nabla}^2_{e_i, Jv}\alpha)(Je_i) - (\widetilde{\nabla}^2_{e_i, v}\alpha)(e_i) = (\widetilde{\nabla}^2_{Jv,e_i}\alpha)(Je_i) - (\widetilde{\nabla}^2_{v, e_i}\alpha)(e_i) - 2\alpha(v) - (2K - 2)[\alpha(v)]^B.
\end{equation}
Moreover, if $(e_1, e_2)$ is geodesic at $p \in \Sigma$, then at the point $p$:
\begin{equation} \label{eq:Db}
J\Psi_{(\nabla^\perp)^*\alpha}(v) = (\widetilde{\nabla}^2_{Jv,e_i}\alpha)(Je_i) - (\widetilde{\nabla}^2_{v, e_i}\alpha)(e_i).
\end{equation}
\end{lem}
\indent Using these identities, we can now give the analogue of Lemma \ref{thm:Codiff1}:

\begin{lem} \label{thm:Codiff2} We have:
$$\widetilde{\nabla}^*\Omega_\eta = -\Psi_{\mathcal{L}\eta} + (2K-2)(\Psi_\eta)^B.$$
\end{lem}

\begin{proof} Let $v \in T\Sigma$.  Using (\ref{eq:CodiffBeta}), followed by the symmetry $(\widehat{\nabla}_Z\Omega_\eta)(X,Y) = (\widehat{\nabla}_Z\Omega_\eta)(Y,X)$, followed by Lemma \ref{thm:B}(c), followed by (\ref{eq:Da}), we get:
\begin{align*}
(\widetilde{\nabla}^* \Omega_\eta)(v) & = -(\widehat{\nabla}_{e_i}\Omega_\eta)(e_i, v) \\
& = -(\widehat{\nabla}_{e_i} \Omega_\eta)(v, e_i) \\
& = -[(\widetilde{\nabla}^2_{e_i, Jv} \Psi_\eta)(Je_i) - (\widetilde{\nabla}^2_{e_i, v} \Psi_\eta)(e_i)] \\
& = -\left[(\widetilde{\nabla}^2_{Jv, e_i} \Psi_\eta)(Je_i) - (\widetilde{\nabla}^2_{v, e_i} \Psi_\eta)(e_i) - 2\Psi_\eta(v) - (2K-2)[\Psi_\eta(v)]^B\right]\!.
\end{align*}
Choose the local frame $(e_1, e_2)$ to be geodesic at $p \in \Sigma$.  By (\ref{eq:Db}), at the point $p \in \Sigma$, we have:
\begin{align*}
J\Psi_{(\nabla^\perp)^*\Psi_\eta}(v) = (\widetilde{\nabla}^2_{Jv,e_i}\Psi_\eta)(Je_i) - (\widetilde{\nabla}^2_{v, e_i}\Psi_\eta)(e_i)
\end{align*}
Using this, together with Lemma \ref{thm:Codiff1}, we conclude that
\begin{align*}
(\widetilde{\nabla}^* \Omega_\eta)(v) & = -J\Psi_{(\nabla^\perp)^*\Psi_\eta}(v) + 2\Psi_\eta(v) + (2K-2)[\Psi_\eta(v)]^B \\
 & = -J\Psi_{-J(\mathcal{L}+2)\eta}(v) + 2\Psi_\eta(v) + (2K-2)[\Psi_\eta(v)]^B \\
& = -\Psi_{\mathcal{L}\eta}(v) + (2K-2)[\Psi_\eta(v)]^B
\end{align*}
which is the result.
\end{proof}

\subsection{Proof of Proposition \ref{thm:GSurj}}

\indent \indent We now prove that $\text{Im}(G) \subset H^0(K_\Sigma^* \otimes L_N)$, which is half of Proposition \ref{thm:GSurj}.  More precisely:

\begin{prop} \label{thm:GSurjA} Let $\eta \in \Gamma(N\Sigma)$.  We have
\begin{align*}
\int_\Sigma \Vert \Omega_\eta \Vert^2 & = 2\int_\Sigma \left\langle \Psi_\eta, \Psi_{\mathcal{L}\eta} + (2-2K)(\Psi_\eta)^B \right\rangle\!.
\end{align*}
In particular, if $\mathcal{L}\eta = 0$ and $(\Psi_\eta)^B = 0$, then $\Psi_\eta$ is real-holomorphic.
\end{prop}

\begin{proof} First, we use the symmetries of $\Omega_\eta$ given by Lemma \ref{thm:B}(b) to observe that
\begin{align*}
\left\langle (\widetilde{\nabla}_{Je_i}\Psi_\eta)(Je_j), \Omega_\eta(e_i, e_j) \right\rangle & =  \left\langle (\widetilde{\nabla}_{e_2}\Psi_\eta)(e_2), \Omega_\eta(e_1, e_1) \right\rangle - \left\langle (\widetilde{\nabla}_{e_2}\Psi_\eta)(e_1), \Omega_\eta(e_1, e_2) \right\rangle \\
& \ \ \ \ \ \ \ -\left\langle (\widetilde{\nabla}_{e_1}\Psi_\eta)(e_2), \Omega_\eta(e_2, e_1) \right\rangle + \left\langle (\widetilde{\nabla}_{e_1}\Psi_\eta)(e_1), \Omega_\eta(e_2, e_2) \right\rangle \\
& =  -\left\langle (\widetilde{\nabla}_{e_2}\Psi_\eta)(e_2), \Omega_\eta(e_2, e_2) \right\rangle - \left\langle (\widetilde{\nabla}_{e_2}\Psi_\eta)(e_1), \Omega_\eta(e_2, e_1) \right\rangle \\
& \ \ \ \ \ \ \ -\left\langle (\widetilde{\nabla}_{e_1}\Psi_\eta)(e_2), \Omega_\eta(e_1, e_2) \right\rangle - \left\langle (\widetilde{\nabla}_{e_1}\Psi_\eta)(e_1), \Omega_\eta(e_1, e_1) \right\rangle \\
& = -\left\langle (\widetilde{\nabla}_{e_i}\Psi_\eta)(e_j), \Omega_\eta(e_i, e_j) \right\rangle\!.
\end{align*}
Using this fact, we can calculate
\begin{align*}
\Vert \Omega_\eta \Vert^2 & = \left\langle \Omega_\eta(e_i, e_j), \Omega_\eta(e_i, e_j) \right\rangle \\
& = \left\langle (\widetilde{\nabla}_{Je_i}\Psi_\eta)(Je_j), \Omega_\eta(e_i, e_j) \right\rangle - \left\langle (\widetilde{\nabla}_{e_i}\Psi_\eta)(e_j), \Omega_\eta(e_i, e_j) \right\rangle \\
& = -2\left\langle (\widetilde{\nabla}_{e_i}\Psi_\eta)(e_j), \Omega_\eta(e_i, e_j) \right\rangle
\end{align*}
and hence
\begin{align*}
\int_\Sigma \Vert \Omega_\eta \Vert^2 = -2 \int_\Sigma \left\langle (\widetilde{\nabla}_{e_i}\Psi_\eta)(e_j), \Omega_\eta(e_i, e_j) \right\rangle & = -2\int_\Sigma \left\langle \Psi_\eta(e_j), (\widetilde{\nabla}^*\Omega_\eta)(e_j) \right\rangle \\
& = 2\int_\Sigma \left\langle \Psi_\eta, \Psi_{\mathcal{L}\eta} + (2-2K)(\Psi_\eta)^B \right\rangle\!,
\end{align*}
where we used Lemma \ref{thm:Codiff2} in the last step.  Finally, note that if $\mathcal{L}\eta = 0$ and $(\Psi_\eta)^B = 0$ both hold, then $\int_\Sigma \Vert \Omega_\eta \Vert^2 = 0$, so that $\Omega_\eta = 0$, so that $\Psi_\eta$ is real-holomorphic. 
\end{proof}

\indent We now make a brief digression.  In general, if $\mathcal{L}\eta = \lambda \eta$, then Proposition \ref{thm:GSurjA} shows that:
\begin{align*}
0 \leq \frac{1}{2}\int_\Sigma \Vert \Omega_\eta \Vert^2 & = \int_\Sigma \left\langle \Psi_\eta, \,\lambda \Psi_{\eta} + (2-2K)(\Psi_\eta)^B \right\rangle \\
& =  \lambda \int_\Sigma \Vert \Psi_\eta \Vert^2 +  \int_\Sigma (2-2K) \left\langle \Psi_\eta, (\Psi_\eta)^B \right\rangle \\
& = \lambda \int_\Sigma \Vert \Psi_\eta \Vert^2 + \int_\Sigma (2-2K) \left\Vert (\Psi_\eta)^B \right\Vert^2.
\end{align*}
This estimate gives:

\begin{prop} \label{thm:Lambda2} Let $u \colon \mathbb{S}^2 \to \mathbb{S}^6$ be a holomorphic $2$-sphere.  If its Gauss curvature $K$ satisfies $K \geq c > 0$, then
$$\lambda_2 \geq -2 + 2c.$$
In particular, the Jacobi operator of the Boruvka sphere satisfies $\lambda_2 \geq -\frac{5}{3}$.
\end{prop}

\begin{proof} Suppose $\eta \in \Gamma(N\Sigma)$ satisfies $\mathcal{L}\eta = \lambda \eta$ with $\lambda > \lambda_1 = -2$.  We estimate
\begin{align*}
0 \leq \frac{1}{2}\int_\Sigma \Vert \Omega_\eta \Vert^2 & = \lambda \int_\Sigma \Vert \Psi_\eta \Vert^2 + \int_\Sigma (2-2K) \left\Vert (\Psi_\eta)^B \right\Vert^2 \\
& \leq \lambda \int_\Sigma \Vert \Psi_\eta \Vert^2 + (2-2c)  \int_\Sigma \left\Vert (\Psi_\eta)^B \right\Vert^2 \\
& \leq (\lambda + 2 - 2c)  \int_\Sigma \left\Vert \Psi_\eta \right\Vert^2.
\end{align*}
If it were the case that $\int_\Sigma \Vert\Psi_\eta\Vert^2 = 0$, then $\Psi_\eta = 0$, so $\eta$ would be real-holomorphic and $\lambda = -2$, contrary to assumption.  Thus, we must have $\int_\Sigma \Vert \Psi_\eta \Vert^2 > 0$, so $\lambda + 2 - 2c \geq 0$, whence the result.
\end{proof}

\begin{rmk} It is proved in \cite{dovv} that a holomorphic curve $u \colon \Sigma^2 \to \mathbb{S}^6$ satisfying $K \geq \frac{1}{6}$ must have either $K \equiv \frac{1}{6}$ or $K \equiv 1$, hence must be either the Boruvka sphere or the totally-geodesic $2$-sphere.  Hence, any non-constant curvature example satisfying the hypothesis of Proposition \ref{thm:Lambda2} must have $c \in (0, \frac{1}{6})$.  I do not know any examples of this type.  We remark that it is also known \cite{dillen87} that the pinching condition $0 \leq K \leq \frac{1}{6}$ implies $K \equiv 0$ or $K \equiv \frac{1}{6}$.  See also \cite{hashimoto2000} for further results.
\end{rmk}

\indent Returning to the main discussion, we now show that $H^0(K_\Sigma^* \otimes L_N) \subset \text{Im}(G)$, thereby completing the proof of Proposition \ref{thm:GSurj}, and hence of Theorem 1.2.

\begin{prop} \label{thm:GSurjB} If $\alpha \in H^0(K_\Sigma^* \otimes L_N)$, then $\alpha = \Psi_\eta$ for some $\eta \in \mathrm{Null}(u)$ with $(\Psi_\eta)^B = 0$.
\end{prop}

\begin{proof} Let $\alpha \in H^0(K_\Sigma^* \otimes L_N)$, and recall the identification
$$H^0(K_\Sigma^* \otimes L_N) \cong \{\alpha \in \Gamma(T^*\Sigma \otimes E_N) \colon \alpha \circ J = -J \circ \alpha \text{ and } \alpha \text{ real-holomorphic}\}.$$
Let $\xi = \frac{1}{2}J(\nabla^\perp)^*\alpha$.  Let $(e_1, e_2)$ be a geodesic frame at $p \in \Sigma$.  Then at $p$, we have, by (\ref{eq:Db}) and (\ref{eq:Da}):
\begin{align*}
\Psi_\xi(v) = \frac{1}{2}J \Psi_{(\nabla^\perp)^*\alpha}(v) & = \frac{1}{2}\left[ (\widetilde{\nabla}^2_{Jv,e_i}\alpha)(Je_i) - (\widetilde{\nabla}^2_{v, e_i}\alpha)(e_i) \right] \\
& = \frac{1}{2}\left[ (\widetilde{\nabla}^2_{e_i, Jv}\alpha)(Je_i) - (\widetilde{\nabla}^2_{e_i, v}\alpha)(e_i) + 2\alpha(v) + (2K - 2)[\alpha(v)]^B \right] \\
& = \frac{1}{2}\left[ (\widetilde{\nabla}^2_{e_i, Jv}\alpha)(Je_i) - (\widetilde{\nabla}^2_{e_i, v}\alpha)(e_i) \right] + \alpha(v)
\end{align*}
where in the last step we used that $[\alpha(v)]^B = 0$.  Finally, using that $\alpha \circ J = - J \circ \alpha$ and that $\alpha$ is real-holomorphic, we have $(\widetilde{\nabla}^2_{e_i, Jv}\alpha)(Je_i) = (\widetilde{\nabla}^2_{e_i, v}\alpha)(e_i)$ at $p \in \Sigma$, whence
$$\Psi_\xi = \alpha.$$
Now, since $\alpha$ is real-holomorphic, it follows that $\Psi_\xi$ is real-holomorphic, so $\Omega_\xi = 0$, so by Lemma \ref{thm:Codiff2}, we have $\Psi_{\mathcal{L}\xi} = 0$, so that $\mathcal{L}\xi$ is real-holomorphic.  Therefore, $\mathcal{L}(\mathcal{L}\xi) = -2\mathcal{L}\xi$, so that $\mathcal{L}(\mathcal{L}\xi + 2\xi) = 0$, and hence $\mathcal{L}\xi + 2\xi = 2\eta$ for some $\eta \in \text{Null}(u)$.  Therefore,
$$\Psi_{2\eta} = \Psi_{\mathcal{L}\xi + 2\xi} = \Psi_{\mathcal{L}\xi} + 2\Psi_\xi = 2\Psi_\xi = 2\alpha,$$
whence $\alpha = \Psi_\eta$ for an $\eta \in \text{Null}(u)$ with $(\Psi_\eta)^B = (\Psi_\xi)^B = 0$. 
\end{proof}

\section{Appendix: Proof of Lemma \ref{thm:D}}

\indent \indent The purpose of this appendix is to prove Lemma \ref{thm:D}, which we restate as Lemma \ref{thm:AppLemma2}.  Throughout, we fix $v \in \Gamma(T\Sigma)$, $\eta \in \Gamma(N\Sigma)$, and a local oriented orthonormal frame $(e_1, e_2)$ on $\Sigma$.

\begin{lem} \label{thm:AppLemma1}
Let $\alpha \in \Gamma(T^*\Sigma \otimes N\Sigma)$. \\
\indent (a) We have:
$$R^\perp_{12}(\eta) = (K-1)J(\eta^N).$$
\indent (b) We have:
\begin{align}
R^\perp(e_i, Jv) \alpha(Je_i) - R^\perp(e_i, v) \alpha(e_i) & = (2K-2)[\alpha(v)]^N \label{eq:RPerpDiff} \\
\alpha(R^\top(e_i, v)e_i) - \alpha(R^\top(e_i, Jv)e_i) & = -2K \alpha(v). \label{eq:RTopDiff}
\end{align}
\end{lem}

\begin{proof} (a) Equations (\ref{eq:RPerp1})-(\ref{eq:RPerp2}) followed by the Gauss equation (\ref{eq:GaussEq}) and the fact that $J(\eta^N) = (J\eta)^N$ give
$$R^\perp_{12}(\eta) = (K-1)(J\eta)^N = (K-1)J(\eta^N).$$
\indent (b) An easy calculation shows that
\begin{align*}
R^\perp(e_i, Jv) \alpha(Je_i) & = R^\perp_{12}[\alpha(Jv)] & R^\top(e_i, v)e_i & = -R^\top_{12}(Jv) = -Kv \\
R^\perp(e_i, v) \alpha(e_i) & = -R^\perp_{12}[\alpha(Jv)] & R^\top(e_i, Jv)Je_i & = R^\top_{12}(Jv) = Kv.
\end{align*}
The equations on the left, together with (a), give (\ref{eq:RPerpDiff}).  The equations on the right give (\ref{eq:RTopDiff}).
\end{proof}

\begin{lem} \label{thm:AppLemma2} Let $\alpha \in \Gamma(T^*\Sigma \otimes N\Sigma)$ satisfy $\alpha \circ J = -J \circ \alpha$. \\
\indent (a) We have:
\begin{align*}
(\widetilde{\nabla}^2_{e_i, Jv}\alpha)(Je_i) - (\widetilde{\nabla}^2_{e_i, v}\alpha)(e_i) & = (\widetilde{\nabla}^2_{Jv,e_i}\alpha)(Je_i) - (\widetilde{\nabla}^2_{v, e_i}\alpha)(e_i) - 2\alpha(v) - (2K - 2)[\alpha(v)]^B.
\end{align*}
\indent (b) If $(e_1, e_2)$ is geodesic at $p \in \Sigma$, then at the point $p$:
\begin{align*}
J\Psi_{(\nabla^\perp)^*\alpha}(v) = (\widetilde{\nabla}^2_{Jv,e_i}\alpha)(Je_i) - (\widetilde{\nabla}^2_{v, e_i}\alpha)(e_i).
\end{align*}
\end{lem}

\begin{proof} (a) Let $L$ denote the left side of the desired identity.  Using the Ricci identity (\ref{eq:RicciIdentity}), followed by the formula (\ref{eq:RTilde}), we have
\begin{align*}
L & = (\widetilde{\nabla}^2_{e_i, Jv}\alpha)(Je_i) - (\widetilde{\nabla}^2_{e_i, v}\alpha)(e_i) \\
& = (\widetilde{\nabla}^2_{Jv, e_i} \alpha)(Je_i) - (\widetilde{\nabla}^2_{v, e_i} \alpha)(e_i) + (\widetilde{R}(e_i, Jv)\alpha)(Je_i) - (\widetilde{R}(e_i, v) \alpha)(e_i) \\
& = (\widetilde{\nabla}^2_{Jv, e_i} \alpha)(Je_i) - (\widetilde{\nabla}^2_{v, e_i} \alpha)(e_i) \\
& \ \ \ \ \ + R^\perp(e_i, Jv)\alpha(Je_i) - R^\perp(e_i, v) \alpha(e_i) + \alpha(R^\top(e_i, v)e_i) - \alpha( R^\top(e_i, Jv)Je_i).
\end{align*}
Now, using equations (\ref{eq:RPerpDiff}) and (\ref{eq:RTopDiff}), we obtain:
\begin{align*}
L & = (\widetilde{\nabla}^2_{Jv, e_i} \alpha)(Je_i) - (\widetilde{\nabla}^2_{v, e_i} \alpha)(e_i) + (2K-2)[\alpha(v)]^N - 2K\alpha(v) \\
& = (\widetilde{\nabla}^2_{Jv, e_i} \alpha)(Je_i) - (\widetilde{\nabla}^2_{v, e_i} \alpha)(e_i) + (2K-2)\alpha(v) - (2K-2)[\alpha(v)]^B - 2K\alpha(v) \\
& = (\widetilde{\nabla}^2_{Jv, e_i} \alpha)(Je_i) - (\widetilde{\nabla}^2_{v, e_i} \alpha)(e_i) - 2\alpha(v) - (2K-2)[\alpha(v)]^B
\end{align*}
This proves (a). \\

\indent (b) Let $(e_1, e_2)$ be a geodesic frame at $p \in \Sigma$.  Then at the point $p$, we have that
$$(\widetilde{\nabla}_{e_i}\alpha)(Je_i) = \nabla^\perp_{e_i}(\alpha(Je_i)) = -J[\nabla^\perp_{e_i}(\alpha(e_i))] = -J[ (\widetilde{\nabla}_{e_i}\alpha)(e_i)].$$
Using this and recalling (\ref{eq:CodiffAlpha}), we compute
\begin{align*}
(\widetilde{\nabla}^2_{Jv, e_i} \alpha)(Je_i) - (\widetilde{\nabla}^2_{v, e_i} \alpha)(e_i) & = (\widetilde{\nabla}_{Jv}\widetilde{\nabla}_{e_i}\alpha)(Je_i) -(\widetilde{\nabla}_v\widetilde{\nabla}_{e_i}\alpha)(e_i) \\
& = -J(\widetilde{\nabla}_{Jv}\widetilde{\nabla}_{e_i}\alpha)(e_i)  - (\widetilde{\nabla}_v\widetilde{\nabla}_{e_i}\alpha)(e_i) \\
& = -J \nabla^\perp_{Jv}[ (\widetilde{\nabla}_{e_i}\alpha)(e_i) ] - \nabla^\perp_v[ (\widetilde{\nabla}_{e_i}\alpha)(e_i) ] \\
& = J \nabla^\perp_{Jv}( (\nabla^\perp)^*\alpha) + \nabla^\perp_v( (\nabla^\perp)^*\alpha) \\
& = J\Psi_{(\nabla^\perp)^*\alpha}(v) 
\end{align*}
which proves the claim. \end{proof}

\bibliographystyle{plain}
\bibliography{FBRef}
\Addresses

\end{document}